\documentclass[10pt]{article}

\pdftrailerid{}
\usepackage[utf8]{inputenc}
\usepackage[T1]{fontenc}
\usepackage{amsmath,amssymb,dsfont}
\numberwithin{equation}{section}
\usepackage{microtype}
\usepackage{graphicx,tikz,pgfplots}
\graphicspath{{images/}}
\pgfplotsset{compat=newest}
\usepackage[hyperref,amsmath,thmmarks]{ntheorem}
\usepackage{aliascnt}
\usepackage[a4paper,centering,bindingoffset=0cm,marginpar=2cm,margin=2.5cm]{geometry}
\usepackage[pagestyles]{titlesec}
\usepackage[font=footnotesize,format=plain,labelfont=sc,textfont=sl,width=0.75\textwidth,labelsep=period]{caption}

\usepackage[backend=biber,maxnames=10,backref=true,hyperref=true,giveninits=true,safeinputenc]{biblatex}
\DefineBibliographyStrings{english}{%
	backrefpage = {cited on page},
	backrefpages = {cited on pages},
}

\DeclareFieldFormat[report]{title}{``#1''}
\DeclareFieldFormat[book]{title}{``#1''}
\AtEveryBibitem{\clearfield{url}}
\AtEveryBibitem{\clearfield{note}}

\titleformat{\section}[block]{\large\sc\filcenter}{\thesection.}{0.5ex}{}[]
\titleformat{\subsection}[runin]{\bf}{\thesubsection.}{0.5ex}{}[.]

\usepackage[pdftex,colorlinks=true,linkcolor=blue,citecolor=green,urlcolor=blue,bookmarks=true,bookmarksnumbered=true]{hyperref}
\hypersetup
{
    bookmarksnumbered,
    breaklinks,
    pdfborderstyle=,
    colorlinks,
    citecolor=[rgb]{0,.65,0},
    linkcolor=[rgb]{0,0,.5},
    urlcolor=[rgb]{0,0,.5},
    pdfauthor={Authors},
    pdfsubject={Subject},
    pdftitle={Title},
    pdfkeywords={Keywords}
}

\newpagestyle{headers}
{
	\headrule
	\sethead[\footnotesize\thepage][\footnotesize\sc Authors][]{}{}{\footnotesize\thepage}
	\setfoot{}{}{}
}
\newtheorem{lemma}{Lemma}[section]

\newaliascnt{proposition}{lemma}

\aliascntresetthe{proposition}

\newaliascnt{corollary}{lemma}

\aliascntresetthe{corollary}

\newaliascnt{theorem}{lemma}
\newtheorem{theorem}[theorem]{Theorem}
\aliascntresetthe{theorem}

\theorembodyfont{\normalfont}
\newaliascnt{definition}{lemma}
\newtheorem{definition}[definition]{Definition}
\aliascntresetthe{definition}

\newaliascnt{assumption}{lemma}

\aliascntresetthe{assumption}

\theorembodyfont{\normalfont}
\newaliascnt{example}{lemma}
\newtheorem{example}[example]{Example}
\aliascntresetthe{example}

\newaliascnt{notation}{lemma}

\aliascntresetthe{notation}

\newaliascnt{convention}{lemma}
\newtheorem{convention}[convention]{Convention}
\aliascntresetthe{convention}

\newaliascnt{algorithm}{lemma}

\aliascntresetthe{algorithm}

\newaliascnt{remark}{lemma}
\newtheorem{remark}[remark]{Remark}
\aliascntresetthe{remark}

\theoremstyle{nonumberplain}
\theoremseparator{:}
\theoremheaderfont{\normalfont\itshape}

\theoremsymbol{\ensuremath{\square}}
\newtheorem{proof}{Proof}

\newcommand{\N}{\mathds{N}}

\newcommand{\R}{\mathds{R}}

\let\RE\Re
\let\Re=\undefined
\DeclareMathOperator{\Re}{\RE e}
\let\IM\Im
\let\Im=\undefined
\DeclareMathOperator{\Im}{\IM m}

\newcommand{\norm}[1]{\left\|#1\right\|}
\newcommand{\set}[1]{\left\{#1\right\}}
\newcommand{\inner}[2]{\left<#1,#2\right>}
\newcommand{\e}{\mathrm e}
\let\ii\i
\renewcommand{\i}{\mathrm i}
\renewcommand{\d}{\,\mathrm d}

\newcommand{\ve}{\varepsilon}

\newcommand{\Ym}{Y_{N}}
\newcommand{\Xm}{X_{N}}
\newcommand{\ol}[1]{{\overline{#1}}}
\newcommand{\ulm}{\underline{m}}
\newcommand{\olm}{\overline{m}}

\usepackage{subcaption}

\title{Spectral Function Space Learning and Numerical Linear Algebra Networks for Solving Linear Inverse Problems}
\author{Andrea Aspri$^{4}$\\{\footnotesize\href{mailto:andrea.aspri@unimi.it}{andrea.aspri@unimi.it}}
\and Leon Frischauf$^{1}$ \\{\footnotesize\href{mailto:leon.frischauf@univie.ac.at}{leon.frischauf@univie.ac.at}}
\and Otmar Scherzer$^{1,2,3}$\\{\footnotesize\href{mailto:otmar.scherzer@univie.ac.at}{otmar.scherzer@univie.ac.at}}}
\date{}
\newcommand{\bx}{{\bf{x}}}
\newcommand{\obx}{\overline{\bx}}

\newcommand{\by}{{\bf{y}}}
\newcommand{\oby}{\overline{\by}}

\newcommand{\bu}{{\bf{u}}}
\newcommand{\bv}{{\bf{v}}}

\newcommand{\ba}{{\bf{a}}}
\newcommand{\oba}{\overline{\ba}}

\newcommand{\bY}{{\mathbf{Y}}}
\newcommand{\bW}{{\mathbf{W}}}
\newcommand{\bX}{{\mathbf{X}}}
\newcommand{\Id}{{\mathcal{I}}}

\begin{document}

\maketitle
\thispagestyle{empty}
\begin{center}
\parbox[t]{17em}{\footnotesize
\hspace*{-1ex}$^1$Faculty of Mathematics\\
University of Vienna\\
Oskar-Morgenstern-Platz 1\\
A-1090 Vienna, Austria}
\hfil
\parbox[t]{17em}{\footnotesize
\hspace*{-1ex}$^2$Johann Radon Institute for Computational\\
\hspace*{1em}and Applied Mathematics (RICAM)\\
Altenbergerstraße 69\\
A-4040 Linz, Austria}
\end{center}

\begin{center}
\parbox[t]{17em}{\footnotesize
\hspace*{-1ex}$^3$Christian Doppler Laboratory\\
\hspace*{1em}for Mathematical Modeling and Simulation\\
\hspace*{1em}of Next Generations of Ultrasound Devices\\
\hspace*{1em}(MaMSi)\\
Oskar-Morgenstern-Platz 1\\
A-1090 Vienna, Austria}
\hfil
\parbox[t]{17em}{\footnotesize
		\hspace*{-1ex}$^4$Università degli Studi di Milano Statale\\
		Department of Mathematics\\ ``Federigo Enriques''\\
		Via Saldini, 50\\
		20133 Milano, Italy}
\end{center}

\begin{abstract}		
	We consider solving a probably ill-conditioned linear operator equation, where the operator is not modeled by physical laws but is specified via training pairs (consisting of images and data) of the input-output relation of the operator. We derive a stable method for computing the operator, which consists of first a Gram-Schmidt orthonormalization of images and a principal component analysis of the data. This two-step algorithm provides a spectral decomposition of the linear operator. Moreover, we show that both Gram-Schmidt and principal component analysis can be written as a deep neural network, which relates this procedure to de-and encoder networks. Therefore, we call the two-step algorithm a linear algebra network. Finally, we provide numerical simulations showing the strategy is feasible for reconstructing spectral functions and for solving operator equations without explicitly exploiting the physical model.
\end{abstract}

\section{Introduction}
We consider solving a probably ill-conditioned {\bf linear operator equation} 
\begin{equation}\label{eq:ip}
	F \bx = \by,
\end{equation}
where $\bx \in \R^{\ulm}$ and $\by \in \R^{\olm}$. We call $\R^{\ulm}$ the {\bf image} and $\R^{\olm}$ the {\bf data space}
following the terminology of \cite{AspKorSch20,AspFriKorSch21}. The main assumption of this work is that the operator $F$ is {\bf not} modeled by physical laws but indirectly via {\bf training pairs}, $\mathcal{P}:=\set{(\bx_i,\by_i):i=1,\ldots,{N}},$ which satisfy
\begin{equation} \label{eq:itp}
	F\bx_i=\by_i \quad i=1,\ldots,{N}.
\end{equation} 
${N}$ is called the sampling size, and we denote the span of the training images and data by 
\begin{equation*}
	\Xm := \text{span}\set{\bx_i:i=1,\ldots,{N}} \subseteq \R^{\ulm}, \quad \Ym := \text{span}\set{\by_i:i=1,\ldots,{N}} \subseteq \R^{\olm}.
\end{equation*}
Without further notice, we always assume that the training images $\bx_i$ are linearly independent and that $F$ has trivial nullspace, such that the training data $\by_i$ are also linearly independent.

In this paper we study {\bf learning} the operator $F$ and its inverse by {\bf en-} and {\bf decoding}, which refer to exclusive use of training data. After learning the operator $F$, we can solve \autoref{eq:ip} for arbitrary data $\by \in \R^{\olm}$.
Operator learning is a very active field of research: There exist a variety of such methods, such as {\bf black box} strategies (see, for instance, \cite{Pap62}) for linear operator learning. For nonlinear operators {\bf deep neural network} learning can be used \cite{KovLanMis21,LanMisKar22,LanLiStu23_report,LanStu23_report}. For applications in inverse problems, see \cite{ArrMaaOktScho19,HalNgu22}. Coding is a term used in {\bf manifold learning}, which, in turn, is a basic tool in machine learning. The basic assumption there is that all potentially measured $\by$'s are elements of a {\bf low-dimensional} manifold (see, for instance, \cite{DufCamEhr24,BraRajRumWir21}). The setting of manifold learning (no operator connecting data) is represented in \autoref{fig:scheme}. 
\begin{figure}[h]
	\centering
	\includegraphics[scale=0.91]{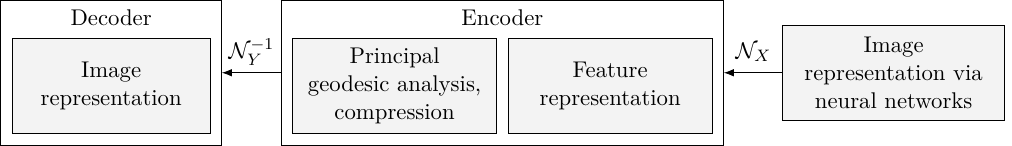}
	\vspace*{-1mm}
	\captionsetup{margin={5mm,5mm}}%
	\caption{\label{fig:scheme} Variational en- and decoding with neural networks: 
		The image data are represented via a neural network. After that they are transformed into a {\bf feature} space (with the operator $\mathcal{N}_{X}$). The features are compressed by a principal geodesic analysis. The decoder $\mathcal{N}_{Y}^{-1}$ (we assume for the sake of simplicity that the operator is invertible) transforms
		features into images. }
\end{figure}
Our strategy for operator learning is conceptually similar to manifold learning but differs by the used techniques (compare \autoref{fig:scheme} and \autoref{fig:scheme_linear}). 
\begin{figure}[h]
	\centering
	\includegraphics[scale=0.9]{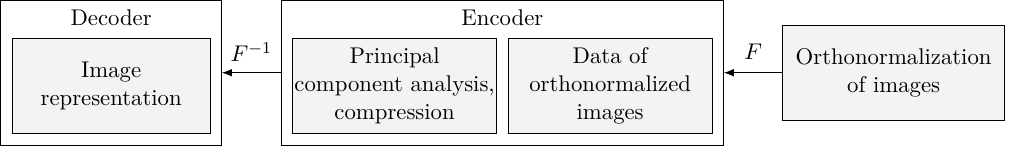}
	\vspace*{-1mm}
	\captionsetup{margin={4mm,4mm}}%
	\caption{\label{fig:scheme_linear} En- and decoding scheme for {\bf linear} operators: First the image data are orthonormalized and the according data is computed by applying $F$ - this is done by explicit calculations without making use on any physical model describing the forward operator. On the orthonormalized data a principal component analysis (PCA)  is applied, which allows to compress the data space. The decoder calculates the inverse of some given data in the compressed space.}
	\vspace*{-2mm}
\end{figure}
In this paper, we investigate coding for solving {\bf linear ill--posed problems} as outlined in \autoref{fig:scheme_linear}. We show that orthonormalization is a key tool for operator learning (this is of course not a new observation, see for instance \cite{EngHoo20}). The paper is based on the following two truly interesting observations:  
\begin{enumerate}
	\item \autoref{fig:scheme_linear} determines the singular vectors of $F$ (see \autoref{th:proc}). In other words, the choice of the training pairs only influences which spectral vectors and values are determined. Related to this result are regularization issues: In \cite{AspKorSch20,AspFriKorSch21}, we showed that orthonormalization of training data or training images (in a separate manner) can be used to stably determine an operator if the orthonormalization produces a basis that is close to the singular vectors. There is an anti-symmetry in the two approaches, which is overcome in this paper. In this context, regularization and stability analysis should not be confused. For instance, as used here, Gram-Schmidt orthonormalization is unstable with an increased number of expert pairs. However, regularization results apply if the pre-processing orthonormalization methods (like Gram-Schmidt) have been implemented in a stable manner. In other words, the truncated singular value decompositions (SVD) is a stable regularization method (presuming that the singular functions have been calculated exactly).
	\item Moreover, we show that each building block of \autoref{fig:scheme_linear} (in particular orthonormalization)can be expressed via a customized {\bf linear algebra network} (see \autoref{fig:zymlk} in \autoref{sec:ortho}). The term customized refers that the parameters in the neural network are given by the algorithm and do not need to be optimized.
\item Finally, we present some numerical experiments on learning the Radon operator (source code \cite{data-learning-op-g-66893bab}). For this operator, the singular values are explicitly known (see \cite{Dav81,Nat01}). Therefore, we can compare the computed singular values from training data with the analytical ones. See \autoref{se:numerics}. Although it is theoretically possible to recover singular vectors, practical handicaps are due to instabilities of the orthonormalization algorithms. 
\end{enumerate}

\section{Encoding of linear operators} \label{sec:encoding} 
The encoder from \autoref{fig:scheme_linear} consists of two steps:
\begin{enumerate}
	\item Calculating orthonormalized images.
	\item Computing a principal component of data of orthormalized images ($\oby_j:=F(\obx_j), j=1,\ldots,N$). 
\end{enumerate}
We recall that according to our general assumptions, training images and data, respectively, and thus ${N} \leq \set{\ulm,\olm}$. This will be assumed in the course of this paper.	
	
\subsection{Pre-processing - orthonormalization} \label{sec:ortho}
We review the Gram-Schmidt orthonormalization method (see for instance \cite{GolVan96}) and show that it can be expressed as a deep neural network (NN). In a second step we show that also QR decomposition, realizing PCA, is representable as a deep NN. In other words, they are {\bf customized} because the parameters do not need to be optimized. Therefore, we call such networks {\bf linear algebra networks}. 

\subsubsection*{Gram-Schmidt as a deep network} 

We start with the training images $\bx_1,\ldots,\bx_{N} \in \R^{\ulm}$ and orthonormalize them iteratively.
Let 
\begin{equation*}\begin{aligned}
		\sigma : \R^{\ulm} \backslash \set{0} & \to \R^{\ulm}.\\
		\bx &\mapsto \frac{\bx}{\norm{\bx}} 
	\end{aligned}
\end{equation*}
Then Gram-Schmidt looks as follows:
\begin{equation} \label{eq:gs}
	\begin{aligned} 
		\obx_j &:= \sigma \biggl(\underbrace{
			{\bx}_j - \sum_{i=1}^{j-1}\inner{{\bx}_j}{\obx_i} \obx_i
		}_{=:\rho({\bx}_j)} \biggr) \text{ for all } j=1,\ldots,{N}.
	\end{aligned}
\end{equation}
\begin{figure}[h]
	\begin{center}
		\includegraphics{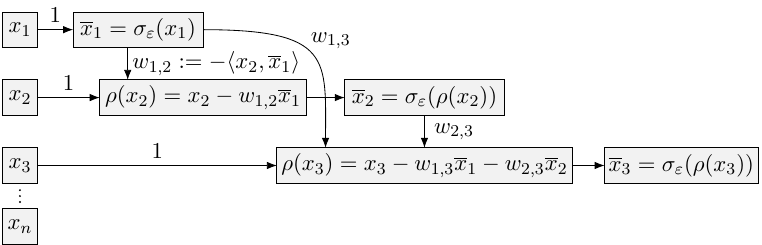}
		\caption{\label{fig:zymlk} The neural network structure of the Gram-Schmidt orthonormalization. We use here the function $\sigma_\ve$ from \autoref{eq:sigmave} as activation function. }
	\end{center}
\end{figure}
\begin{remark}
	In machine learning context instead of the high-dimensional signum function $\sigma$ the smooth approximation
	\begin{equation} \label{eq:sigmave} 
		\begin{aligned}
		 \sigma_\ve:\R^{\olm} &\to \R^{\olm}\\ 
		                  \bx         &\mapsto \frac{\bx}{\sqrt{\norm{\bx}^2+\ve^2}}
		\end{aligned} 
	\end{equation}
    is used and called {\bf activation function}. Replacing in \autoref{fig:zymlk} $\sigma$ by $\sigma_\ve$ we obtain an ${N}$-layer neural network, which we call {\bf Gram-Schmidt network}.
	We note, however, that this is not a standard neural network (see, for instance, \cite{Bis95}) because there $\hat{\sigma}_\ve:\R^1 \to \R^1$  is evaluated for each component of $\R^N$. Such network cannot be used here because we must ensure that every $\rho(\bx_j)$, $j=1,\ldots,{N}$ is a linear combination of $\bx_i$, $i=1,\ldots,j-1$. There are efficient alternatives to Gram-Schmidt, which are, for instance, block based (see, for instance, \cite {CarLunRozTho22}), which can be reinterpreted again as deep neural networks. Gram-Schmidt breaks down if and only if one of the vectors $\rho(\bx_j)$ becomes zero, or in other words, if the vectors $\bx_j$, $j=1,\ldots,{N}$ become linearly dependent. We excluded this by a general assumption that all images are linearly independent. With a smooth approximation $\sigma_\ve$, the linear dependence of the training images does not result in a break down of the algorithm described in \autoref{fig:zymlk}. In fact, all available training pairs can be used. However, the linear dependence assumption cannot be avoided for the analysis.
\end{remark}

\subsection{Data from orthonormalized images} 
If the nullspace of $F$ is trivial, then from \autoref{eq:gs} it follows immediately  that (see \autoref{fig:zymlk}) that
\begin{equation} \label{eq:gsi}
	\begin{aligned} 
		\oby_j := F\obx_j := F \left( \sigma (\rho({\bx}_j)) \right) = \frac{1}{\norm{\rho({\bx}_j)}}  F \left( \rho({\bx}_j) \right)		                                                         
		\text{ for all } j=1,\ldots,{N}.
	\end{aligned}
\end{equation}
This means that the data of orthonormalized images can be computed without explicit knowledge of $F$. Note that $\rho({\bx}_j)$ is a linear combination of ${\bx}_i$, $i=1,\ldots,j$.

With the knowledge of $\oby_j$, $j=1,\ldots,{N}$, we can compute data for every image $\bx \in \Xm$. Since 
	\begin{equation*}
		\bx = \sum_{i=1}^N \inner{\bx}{\obx_i} \obx_i, 
	\end{equation*}
    it follows that 
	\begin{equation*}
		\by = F	\bx = \sum_{i=1}^N \inner{\bx}{\obx_i} F \obx_i = \sum_{i=1}^N \inner{\bx}{\obx_i} \oby_i.
	\end{equation*}

\begin{remark} If instead of an exact Gram-Schmidt the deep network Gram-Schmidt from \autoref{fig:zymlk} is used, we can apply the approximate formula with $\sigma_\ve$ from \autoref{eq:sigmave} is used
	\begin{equation} \label{eq:gsi2}
		\begin{aligned} 
			\oby_j := F\obx_j \approx F \left( \sigma_\ve (\rho({\bx}_j)) \right) =: \frac{1}{\norm{\sigma_\ve \left( \rho({\bx}_j) \right)}}  F \left( \rho({\bx}_j) \right)		                                                         
			\text{ for all } j=1,\ldots,{N}.
		\end{aligned}
	\end{equation}
\end{remark}

\subsection*{Principal component analysis (PCA)}
In the following, we analyze under which assumptions the PCA is stable. Starting point of this discussion is {\bf Seidman's veto} \cite{Sei80}, which states that regularization by projection is in general not a regularization method. This means that by projecting onto $X_S$ inversion of $F$ on the range of $F$ of $X_S$ is not stable. On the other hand if $F$ is inverted on $Y_S$, then it is in fact stable (see again \cite{Sei80}). However, in the context of machine learning this requires to collect training data of $F^*\by_i$, $i=1,\ldots,S$, which is in general not available. As a conclusion from \cite{Sei80} we find that $F$ can be stably inverted by projection if $X_S$ is the space of the singular vectors corresponding to the largest $S$ singular values (see for instance \cite{EngHanNeu96}). We emphasize again, that here stability refers to inversion of the truncated SVD after orthonormalization. Now, we show how this singular values can be calculated with machine learning techniques:

\begin{theorem}[Spectral theory: See Theorem 2.5.2 in \cite{GolVan96}] \label{th:svd} Let the operator $F: \R^{\ulm} \to \R^{\olm}$ be linear. Then for every $\bx \in \R^{\ulm}$ 

	\begin{equation} \label{eq:svd}
		F \bx = \sum_{j=1}^{\min \set{\ulm,\olm}} \gamma^j \inner{\bx}{\bu^j} \bv^j \text{ and } F^T \bv^j = \gamma^j \bu^j,
		F \bu^j = \gamma^j \bv^j,
	\end{equation}
	where $\bu^j \in \R^{\ulm}$, $j=1,\ldots,\ulm$ and $\bv^j \in \R^{\olm}$, $j=1,\ldots,\olm$ are orthonormal, respectively, and 
	\begin{equation*}
		0 \leq \gamma^1 \leq \gamma^2 \leq \cdots \leq \gamma^{\min \set{\ulm,\olm}}.
	\end{equation*}
	In matrix form, this identity becomes more compact:
	\begin{equation} \label{eq:svdd}
		F = {\bf V} D {\bf U}^T \text{ with } {\bf U} \in \R^{{\ulm \times \ulm}}, {\bf V} \in \R^{{\olm \times \olm}},
	\end{equation}
	where ${\bf U}$ and ${\bf V}$ are orthonormal and 
	\begin{equation*}
		\begin{aligned}
			D = \text{diag} (\gamma^1,\gamma^2, \cdots,\gamma^{\min \set{\ulm,\olm}}) \in \R^{\olm \times \ulm}.
		\end{aligned}
	\end{equation*}
\end{theorem}
In the following, we group the spectral values: 
\begin{definition} Let $F$ be linear with trivial nullspace with spectral decomposition as in \autoref{eq:svd}. We denote by 
	\begin{equation}
		\Gamma := \set{\gamma^j : j =1,\ldots, \min\set{\ulm,\olm}} = \set{\hat{\gamma}^k : k=1,\ldots,\hat{m}} 
	\end{equation}
    the set of distinct singular values. Because $F$ is assumed to have trivial nullspace $\gamma^j >0$ for all $j =1,\ldots, \min\set{\ulm,\olm}$. Moreover, we associate to each multiple singular value the associated singular vectors:
	For every $k=\set{1,\ldots,\hat{m}}$ let 
	\begin{equation}\label{eq:Ek}
		E^k := \text{span} \set{\bv^j : \gamma^j = \hat{\gamma}^k, j=1,\ldots, \min \set{\ulm,\olm}}.
	\end{equation}
\end{definition}
Now, we apply the spectral theory to show that the proposed decoding algorithm is stable. 
There, we make some general notation: 
\begin{definition} Let $\obx_i$, $i=1,\ldots,N$ be the orthonormalized training images and $\oby_i$, $i=1,\ldots,N$ the according images as defined in \autoref{eq:gs} and \autoref{eq:gsi}, respectively. We denote:
	\begin{equation}\label{eq:pcaa}
		{\ol \bX} := (\obx_1,\ldots,\obx_{N}) \in \R^{{\ulm \times {N}}}, \quad 
		{\ol \bY} = (\oby_1,\ldots,\oby_{N}) \in \R^{{\olm \times {N}}} \text{ and } A:={\ol \bY} {\ol \bY}^T \in \R^{{\olm \times \olm}}.
	\end{equation} 
Note that by our general assumptions, the rank of each of the three matrices is always $N \leq \olm$.
\end{definition}
With this result, we can state the main result of this paper:
\begin{theorem} \label{th:proc} Let $\psi_j \in \R^{\olm}$, $j=1,\ldots,{N}$ be a non-zero eigenvector of $A$, then 
	\begin{enumerate}
		\item there exists $k \in \set{1,\ldots,\hat{m}}$ such that $\psi_j \in E^{k}$. 
		\item This, in particular, means that if a non-zero eigenvector of $A$ has multiplicity one, then it equals a spectral vector of $F$.
	\end{enumerate} 
\end{theorem}
\begin{proof}
	The PCA in image space calculates the eigenvalue decomposition of the covariance matrix 
	\begin{equation*} 
		A:={\ol \bY} {\ol \bY}^T = {\ol \bW} \underbrace{\begin{bmatrix} \Lambda \in \R^{{N} \times {N}} & 0\\0 & 0 \end{bmatrix}}_{=: \ol{\Lambda}} {\ol \bW}^T \in \R^{{\olm \times \olm}},
	\end{equation*}
	where ${\ol \bW} \in \R^{{\olm \times \olm}}$ is an orthonormal matrix ${\ol \bW}$ describes the 
	principal directions of the data ${\ol \bY}$, and the according entries of $D$ describe the elongation of the data in this direction. Now, since ${\ol \bY} = F {\ol \bX}$, we get
	\begin{equation} \label{eq:ortho}
		F {\ol \bX} {\ol \bX}^T F^T = {\ol \bY} {\ol \bY}^T = {\ol \bW} D {\ol \bW}^T.
	\end{equation}
	Since ${\ol \bX}$ consists of orthonormal vectors and has rank ${N}$ we have
	\begin{equation*}
		\ol{\bX} \ol{\bX}^T = \begin{bmatrix} \Id \in \R^{{N} \times {N}} & 0\\0 & 0 \end{bmatrix} =: \ol{\Id} \in \R^{\ulm \times \ulm},
	\end{equation*}
	where $\Id$ is the unitary matrix. 	Let ${\bf U} \in \R^{\ulm \times \ulm}$ be orthonormal as in \autoref{eq:svdd}, then, after reordering of columns of $\ol{\bX}$,  
	\begin{equation*}
		{\bf U}^T\ol{\bX}^T \ol{\bX} {\bf U} = {\bf U}^T \ol{\Id} {\bf U} \in \R^{\ulm \times \ulm}.
	\end{equation*}
	The matrix $\ol{\Id}{\bf U} $ projects onto the first basis vectors $\bu_i$ 
	(after reordering).
	Then we get from \autoref{eq:ortho} and \autoref{eq:svdd} the identity
	\begin{equation} \label{eq:ortho2}
		{\ol \bW} \ol{\Lambda} {\ol \bW}^T = {\ol \bY} {\ol \bY}^T =
		F {\ol \bX} {\ol \bX}^T F^T = 
		{\bf V} D \ol{\Id} D^T {\bf V}^T.		 
	\end{equation}
	
	This means that we have found two singular value decompositions of ${\ol \bY} {\ol \bY}^T$. We know that the matrices on the left and right have the same eigenspaces. This means, in particular, that if an eigenvalue has multiplicity one, then the according eigenspaces of $\bf V$ and $\bW$ match, and the the representing eigenvectors of the eigenspace are identical up to sign. This proves the second item. For eigenvalues of multiplicity higher than one, the corresponding columns of ${\bf V}$ are rotations and mirrors of ${\ol \bW}$, which proves the first item. 
\end{proof}
\begin{remark}
	\autoref{th:proc} tells us 
	\begin{enumerate}
		\item that orthonormalization of the training images is the basis of stable decoding. Problematic is, however, that orthonormalization may not provide us with an ordering of the singular vectors. For instance, there is no guarantee 
		that singular vectors can be recovered according to low indices, which usually carry most of the information.
		\item Orthonormalization is unstable. Thus, ${N}$ must be chosen small to guarantee numerical stability.
	\end{enumerate}
\end{remark}
Last, we also verify that PCA can be approximately rewritten as a network, analogously as the Gram-Schmidt algorithm in \autoref{fig:zymlk}. We note that the QR decomposition is used to compute the PCA. 

\subsection{PCA expressed as a deep network} \label{ss:pca}

This subsection shows that spectral value decomposition can be expressed as a customized deep network. Thereby, we make use of the fact that {\bf singular value decomposition (SVD)} applied to the covariance matrix $A ={\ol \bY} {\ol \bY}^T \in \R^{{\olm \times \olm}}$ from \autoref{eq:svd} gives the principal components. In other words it is a realization of the PCA. The SVD can be implemented by iterative application of the $QR$-algorithm, which is the famous {\bf Francis algorithm} \cite{Fra61} (see also \cite{Wil71,GolVan96}). Now, we show that the $QR$ can be approximated by a deep linear algebra network, and in turn approximating the $QR$ decomposition in the Francis algorithms with linear algebra networks, gives a {\bf ``deep-deep''} linear algebra network realizing the PCA.
  
Let us denote the column vectors of $A$  by ${\bf a}_i$, $i=1,\ldots,\olm$ and denote the orthonormalized vectors by ${\ol{\bf a}_i}$, $i=1,\ldots,\olm$. Moreover, let
\begin{equation*}
	\rho({\ba}_j):= {\ba}_j - \sum_{i=1}^{j-1}\inner{{\ba}_j}{\oba_i} \oba_i.
\end{equation*}
Then it follows (see \cite{Ste05}, which is the elementary exposition where the formulas are used precisely in the same way as they are used here):
\begin{equation}\label{eq:ba}
	{\bf a}_j = \norm{\rho({\bf a}_j)}\ol{{\bf a}}_j+\sum_{i=1}^{j-1} \inner{{\bf a}_j}{\ol{\bf a}_i}\ol{\bf a}_i, \quad j=1,\ldots,\olm.
\end{equation}
Writing this in matrix notation reads as follows:
\begin{equation}
	A=QR \text{ with } Q=(\ol{\bf a}_1,\ldots,\ol{\bf a}_n) \text{ and } R_{i,j} = 
	\left\{ 
	\begin{array}{rl} 
		\inner{{\bf a}_j}{\ol{\bf a}_i} & \text{ for } i=1,\ldots,j-1 \\ 
		\norm{\rho({\bf a}_j)} & \text{ for } i=j\\
		0 & \text{ for } i=j+1,\ldots,\olm
	\end{array} 
	\right. \;.
\end{equation}
In fact, it follows from \autoref{eq:ba}: 
\begin{equation}\label{eq:baa}
	\ol{{\bf a}}_j = \frac{1}{\norm{\rho({\bf a}_j)}} \left( {\bf a}_j - \sum_{i=1}^{j-1} \inner{{\bf a}_j}{\ol{\bf a}_i}
	\ol{\bf a}_i \right) = \sigma (\rho(\ba_j)), \quad j=1,\ldots,\olm,
\end{equation}
\autoref{fig:zymlk} with $\bx$ replaced by $\ba$ is the net, which determines $Q$. This shows that the matrix $Q$ from the $QR$-algorithm can be determined with a {\bf linear algebra network}.  Moreover, it has been recently observed that matrix multiplication can be implemented very efficiently via reinforcement learning or, in other words, via nets (see \cite{FawBalHuaHubRom22}). The demystification is that tensor products are prime examples of fully connected networks. We mention that already Wilkinson \cite[Sect. 35]{Wil88} stated that writing $QR$ as Gram-Schmidt orthonormalization might lead to unstable implementations. Thus, the exposition above is only used for theoretical purposes.

\begin{remark}
	In the last years PCA-networks have become popular in the machine learning community (see \cite{BhaHosKovStu20_report,TanMicRadKuhMaa24,NelStu24}). The goal there is somehow different and consists in minimizing the parameters of a network. Such networks could indeed be useful to investigate here on top, but this is beyond the scope of this paper. 
\end{remark}

\section{Image representation and decoder} \label{sec:decoderII}

We consider the decoder as represented in \autoref{fig:scheme_linear}. We calculate the least-squares approximation of $\by$ with respect to the basis $\set{\psi_l : l=1,\ldots,{N}}$ (see \autoref{th:proc}) of eigenvectors of $A$. This is given by 
\begin{equation} \label{eq:lqs}
	\by_{ls}= \sum_{l=1}^N \inner{\by}{\psi_l} \psi_l.
\end{equation}
The goal is to find an explicit representation of the {\bf decoding function} $\bx_{ls}$. That is a function with minimal norm, which satisfies 
\begin{equation}\label{eq:least}
	F \bx_{ls} = \by_{ls} 
\end{equation}
and which can be represented by the training pairs alone. 

From \autoref{th:proc} it follows that for every $\psi_l$ there exists $k(l) \in \N$ such that 
\begin{equation*} 
	 \psi_l \in E^{k(l)}.
\end{equation*}
We denote the dimension of $E^{k(l)}$ with $m(l)$. Or, in other words, the multiplicity of $\gamma_l$ is $m(l)$. 
From \autoref{eq:Ek} we know that 
$$ \psi_l \in \text{span} \set{\bv^{l(j)}: j=1,\ldots,m(l)} = E^{k(l)}.$$
In other words 
\begin{equation} \label{eq:ident}
	\psi_l = \sum_{j=1}^{m(l)} \nu_j \bv^{l(j)}.
\end{equation}
Applying \autoref{eq:svd} and \autoref{eq:Ek} then implies that 
\begin{equation*} 
	\psi_l = \sum_{j=1}^{m(l)} \nu_j \bv^{l(j)} = \frac{1}{\hat{\gamma}^{k(l)}} F \left(\sum_{j=1}^{m(l)} \nu_j \bu^{l(j)} \right).
\end{equation*}

We summarize the calculations now in the following lemma:
\begin{lemma}[Decoder] \label{le:decoder}  The minimum norm solution of \autoref{eq:least} is   
	given by 
	\begin{equation} \label{eq:decoder} \begin{aligned}
		x_{ls} = \sum_{l=1}^N \inner{\by}{\psi_l} \left(\frac{1}{\hat{\gamma}^{k(l)}}\sum_{j=1}^{m(l)} \nu_j \bu^{l(j)} \right) = \sum_{l=1}^N \frac{\inner{\by}{\psi_l}}{(\hat{\gamma}^{k(l)})^2} F^* \left(\sum_{j=1}^{m(l)} \nu_j \bv^{l(j)} \right) = \sum_{l=1}^N \frac{\inner{\by}{\psi_l}}{(\hat{\gamma}^{k(l)})^2} F^* \psi_l,
	\end{aligned} \end{equation}
    and therefore is a {\bf decoder}. 
\end{lemma}
\begin{remark} \autoref{le:decoder} is implemented as follows. $QR$ decomposition is implemented to get the eigenfunction $(\psi_l)$, from which a decoder is implemented via \autoref{eq:decoder}. 
\end{remark}
\section{Numerical simulations}  \label{se:numerics}
We consider two simulation scenarios: First we prove that the singular value decomposition from training pairs is feasible.
Secondly, we study a reconstruction test with a learned operator. In both cases, we use the Radon transform for two-dimensional images as the physical model for the forward operator, leveraging its analytically known singular value decomposition. (see \cite{Dav81,Nat01}).

\subsection{SVD from training pairs}
We verify numerically that orthonormalization of data ($\bx_i \to \overline{\bx}_i$, $i=1,\ldots,N$) and a principal component analysis of $\set{\overline{\by}_i = F(\overline{\bx}_i):i=1,\ldots,N}$ (see \autoref{eq:gsi} and \autoref{eq:gsi2}) provides us with the singular value decomposition of the discrete operator $F$ (see \autoref{fig:scheme_linear}). 

We take as a prototype example for $F$ the Radon transform in two dimensions (see \autoref{de:radon} below). The singular value decomposition of the Radon transform has been computed in \cite{Dav81} (see also \cite{Nat01}) for the general case of images of $n$ variables. In particular for $n=2$ we consider the Radon transform as an operator from $L^2(\mathcal{B}(0,1))$ into $L^2(Z,w^{-1})$, where $\mathcal{B}(0,1)$ is the unit disk in $\R^2$ centered at the origin, $Z=\mathbb{S}^1 \times [-1,1]$, and $w(s)=(1-s^2)^{1/2}$ is a weight function. This means that $L^2(Z,w^{-1})$ is the weighted $L^2$-space with norm
\begin{equation*}	
	\|g\|_{L^2(Z,w^{-1})}=\int_{\mathbb{S}^1}\int_{-1}^1 |g(\omega,s)|^2 w(s)^{-1}\ ds\d\omega \;.
\end{equation*}  
for $g \in L^2(Z,w^{-1})$.  
Before diving into the numerical details, we provide the necessary notational clarifications and preliminary information to facilitate the analysis.
\begin{convention} \label{co:omegaphi}
	For $\vec{\omega} = (\omega_1,\omega_2)^T \in \mathbb{S}^1$ we use the angular representation  $\phi = \tan^{-1} (\omega_2/\omega_1) \in [0,2\pi)$, which therefore satisfies
	\begin{equation} \label{eq:theta}
		\omega_1=\cos(\phi) \text{ and } \omega_2= \sin(\phi). 
	\end{equation}
	Moreover, the orthogonal vector related to $\vec{\omega}$ is given by 
	$\vec{\omega}^\bot = (-\sin(\phi),\cos(\phi))^T$.
\end{convention}
For the reader's convenience we recall the definition of the Radon transform for functions on the unit disk in $\R^2$:
\begin{definition}[Radon transform on the unit disk] \label{de:radon}
	Let $Z:= \mathbb{S}^1 \times [-1,1]$. We define the Radon transform as
	\begin{equation} \label{eq:radon} 
		\begin{aligned}
			R: L^2(\mathcal{B}(0,1)) &\to L^2({Z},w^{-1}).\\
			f &\mapsto R[f](\vec{\omega},s) = \int_{-\sqrt{1-s^2}}^{\sqrt{1-s^2}} f\left( s \vec{\omega} + 
			t \vec{\omega}^\bot \right) d t
		\end{aligned}
	\end{equation}
\end{definition}
One can show that this operator is continuous and satisfies (see \cite{Dav81}): 
\begin{equation*}
	\|R[f]\|_{L^2({Z},w^{-1})}\leq \sqrt{4\pi}\|f\|_{L^2(\mathcal{B}(0,1))}.
\end{equation*}
In the following, we recall the expression of the adjoint operator of the Radon-transform $R$ defined on the unit disk. This is the operator $R^*: L^2(Z,w^{-1}) \to L^2(\mathcal{B}(0,1))$, which satisfies   
\begin{equation*}
\inner{R[f]}{g}_{L^2({Z},w^{-1})} = \inner{f}{R^*[g]}_{L^2(\mathcal{B}(0,1))}.
\end{equation*}
\begin{definition}[Adjoint] \label{de:back}
	For every $g \in L^2(Z,w^{-1})$ and almost all $\vec{x}= (x_1,x_2)^T \in \R^2$ the adjoint of the Radon-transform $\R^*$ is given by
	\begin{equation*}
		R^*[g](\vec{x})=\int_{\mathbb{S}^1}\frac{g(\vec{\omega},\vec{x}\cdot\vec{\omega})}{w(\vec{x}\cdot\vec{\omega})}\, d\vec{\omega}.
	\end{equation*} 
\end{definition}
In the following we give a survey on the singular value decomposition (SVD) of the Radon transform for functions on the unit disk. The results are essential taken from \cite{Nat01} with the main difference that the we consider the adjoint $R^*$ restricted to the range of $R$ (in contrast to \autoref{de:back}). 

Before recalling the spectral decomposition of the Radon-transform, we review the general definition of a spectral decomposition:
\begin{definition}[Spectral decomposition, \cite{EngHanNeu96}] Let $K : X \to Y$ be a compact linear operator. A {\bf singular system} $(u_k,v_k;\gamma_k)$ is defined as follows:
	\begin{enumerate}
		\item $\gamma_k^2$, $k \in \N_0$ are the non-zero eigenvalues of the selfadjoint operator $K^*K$ (and also $KK^*$) written in decreasing order. We always take $\gamma_k > 0$.
		\item $u_k$, $k \in \N_0$ are a complete orthonormal system of eigenvectors of $K^*K$ (on the space $\overline{\mathcal{R}(K^*)} =\overline{\mathcal{R}(K^*K)}$. 
		\item $v_k := \frac{1}{\norm{Ku_k}} Ku_k$, $k \in \N_0$.
	\end{enumerate}
	The set $\set{v_k:k \in \N_0}$ are a complete orthonormal system of eigenvectors of $K K^*$ which span $\overline{\mathcal{R}(K)} = \overline{\mathcal{R}(K K^*)}$. Moreover, the following formulas hold:
	\begin{equation*}
		\begin{aligned}
			K u_k &= \gamma_k v_k, \quad 
			K^* u_k = \gamma_k v_k, \\
			K x &= \sum_{k=0}^\infty \gamma_k \inner{x}{u_k} v_k \text{ for all } x \in X, \quad  	
			K^* y &= \sum_{k=0}^\infty \gamma_k \inner{y}{v_k} u_k \text{ for all } y \in Y. 
		\end{aligned}
	\end{equation*} 
\end{definition}

\begin{theorem}\cite[p 99]{Nat01} \label{th:singular_values}
	The spectral decomposition of the Radon transform is given by $\set{(u_{k,l},v_{k,l};\gamma_{k,l}): (k,l) \in \mathcal{I}}$
	where 
	\begin{enumerate}
		\item \label{it.1} $\mathcal{I} = \set{(k,l): k \in \N_0, l \in \set{0,1,\ldots,k},  \text{ satisfying } l+k \text{ is even}}.$
		\item \label{it.2} $\gamma_k^2=\gamma_{k,l}^2 = \frac{4\pi}{k+1} > 0$ is independent of $l$.
		\item Let $Y_s$ be the spherical harmonics in $\R^2$ and let 
		\begin{equation*}
			s \mapsto C_k(s)=\frac{\sin \left((k+1)\arccos{(s)}\right)}{\sin(\arccos{(s)})} 
		\end{equation*}
		denote the Chebyshev polynomials of the second kind. Moreover, let 
		$$\frac{1}{c(k)} = \norm{ w C_k Y_{k-2l}}_{L^2(Z,w^{-1})} = \norm{\sqrt{w} C_k}_{L^2(-1,1)}.$$
		Then the normalized eigenfunctions of $RR^*$ on the orthogonal complement of the nullspace of $R^*$, $\mathcal{N}^\bot$ are given by 
		\begin{equation}\label{eq:vkl_ukl}
			(s,\omega) \to v_{k,l}(s,\omega):=c(k) w(s)C_k(s)Y_{k-2l}(\omega)  \text{ and } u_{k,l} = \frac{1}{\gamma_k} R^*[v_{kl}]\;. 
		\end{equation}
	\end{enumerate}
\end{theorem}
The calculation of the singular value decomposition of the Radon transform has been computed in several work (see for instance \cite{Dav81,Maa87,Nat01} and \cite[Theorem 6.4]{NatWue01} to mention but a few). 
\begin{remark}\label{ex:sets_Ek} 
	\begin{enumerate}
		\item The existence of a singular value decomposition with $\gamma_k \to 0$ in particular shows that the Radon-transform is compact.
		\item 
		From \autoref{th:singular_values}, \autoref{it.1} and \autoref{it.2} we see that, in general, $\gamma_k$ has multiplicity higher than one. The sets $E^k$ associated to a spectral value $\gamma_k$ are spanned by the spectral functions $v_{kl}$ with $l$ such that $(k,l) \in \mathcal{I}$. For example, taking $k=0,\ldots,7$, and $l=0,\ldots,k$, with $k+l$ even (meaning that $(k,l) \in \mathcal{I}$), we have 
		\begin{equation}\label{eq:Ek_example}
			\begin{aligned}
				E^1&=\textrm{span}\set{v_{0,0}},\quad E^2=\textrm{span}\set{v_{1,1}},\quad E^3=\textrm{span}\set{v_{2,0},v_{2,2}},\quad E^4=\textrm{span}\set{v_{3,1},v_{3,3}},\\
				E^5&=\textrm{span}\set{v_{4,0},v_{4,2}, v_{4,4}},\quad E^6=\textrm{span}\set{v_{5,1},v_{5,3}, v_{5,5}}, \quad E^7 =\textrm{span}\set{v_{6,0},v_{6,2}, v_{6,4}, v_{6,6}}, \\
				E^8 &=\textrm{span}\set{v_{7,1},v_{7,3}, v_{7,5}, v_{7,7}}.
			\end{aligned}
		\end{equation}
		It is an interesting fact stated in \cite{Nat01} that the adjoint of the Radon-transform $R^*$ has a nullspace. In fact the nullspace of $\R^*$ is given by 
		\begin{equation*}
			\mathcal{N} = \set{v_{k,l}: k \in \N_0, l \in \set{0,1,\ldots,k}, \text{ satisfying } l+k \text { is odd}}.
		\end{equation*}
		In fact also $RR^*$ has the same nullspace, which is shown for instance in \cite[p. 99]{Nat01}. 
	\end{enumerate}
	
\end{remark}

\subsection{Numerical simulations} \label{numerics} 
In this section, we present two numerical experiments: 
\begin{enumerate}
	\item The first tests concern computing the SVD of the Radon-transform in two test scenarios:
	\begin{enumerate}
		\item {\bf Learning spectral functions from analytical spectral functions:} We use as training data $(\bx_i,\by_i)$, $i=1,\ldots,49$ the first 49, analytical given, spectral functions $(u_{kl},v_{kl})$, defined in \autoref{eq:vkl_ukl}. We compare the outcome with our orthonormalization approach. 
		\item {\bf Learning spectral functions from arbitrary test functions:} Here we use annotations of images, which are orthonormalized. This is of course a much harder problem.
	\end{enumerate}
	\item We are concerned with decoding by making use of \autoref{eq:decoder}. Here we use both singular value decompositions obtained in the first step.
\end{enumerate}

\subsubsection{Singular value decomposition}
We visualize the first twenty benchmark images, which are the singular functions $u_{k,l}$, $v_{k,l}$ with $(l,k) \in \mathcal{I}$ of the Radon-transform outlined in \autoref{th:singular_values}, \autoref{fig:ex_func_v} and \autoref{fig:ex_func_u}. These plots are used to visualize the difference between some of the benchmark data and the learned singular functions. We emphasize that in the learned approaches we reconstruct vectors $u_{k,l} \in \R^{\underline{m}}$ and $v_{k,l} \in \R^{\overline{m}}$, with $\underline{m}=\overline{m}=2500$. However, we maintain the same notation used in the theoretical part. 
\begin{figure}[h!]
	\centering
	\includegraphics[width=\textwidth]{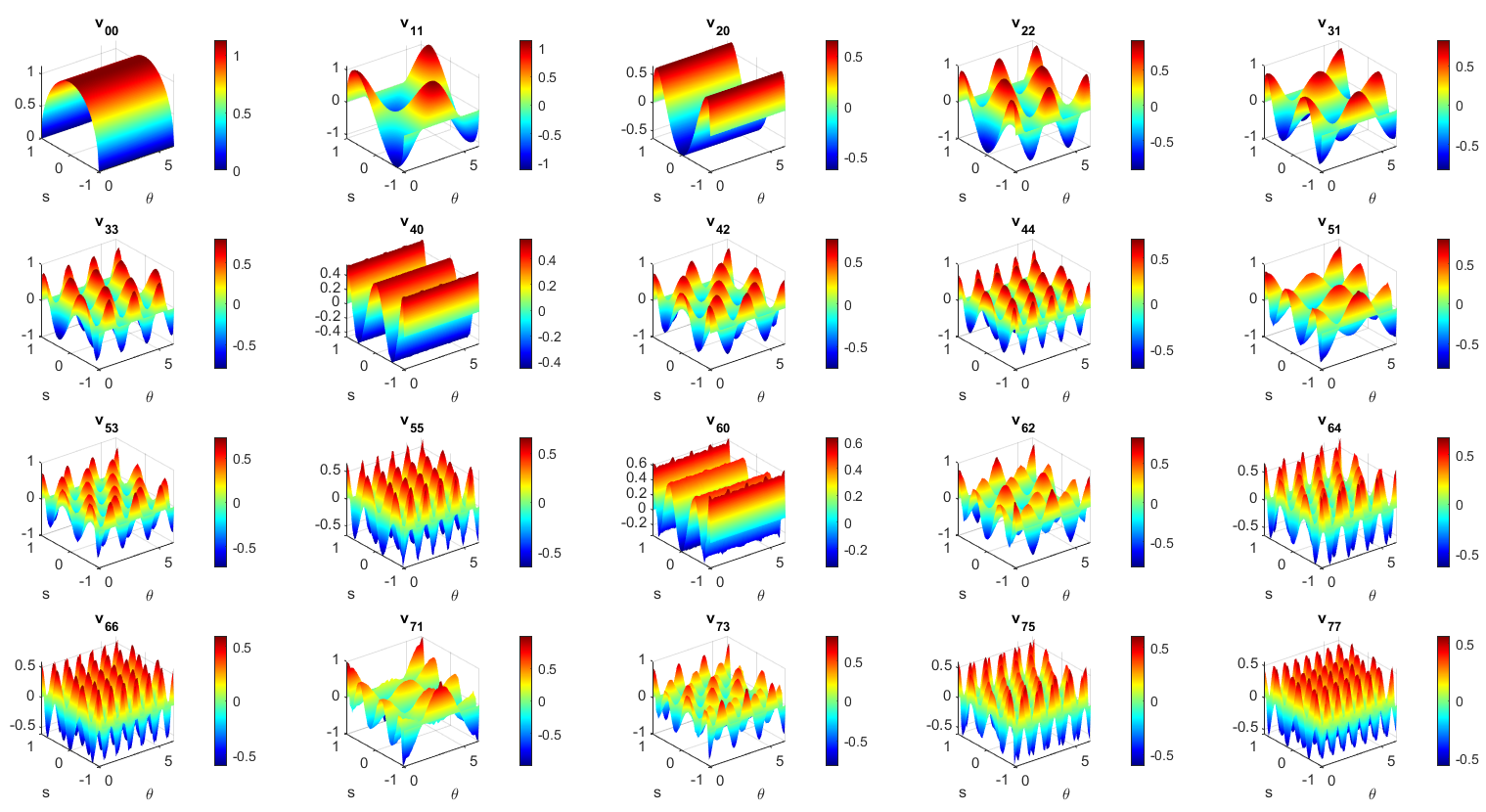}
	\caption{This plot visualizes the first twenty normalized singular functions $v_{kl}$, defined in \autoref{eq:vkl_ukl}, for $k =0,1,\ldots,7$, $l=0,1,\ldots,k$ and $l+k$ even.}
	\label{fig:ex_func_v}
\end{figure}
\begin{figure}[h!]
	\centering
	\includegraphics[width=\textwidth]{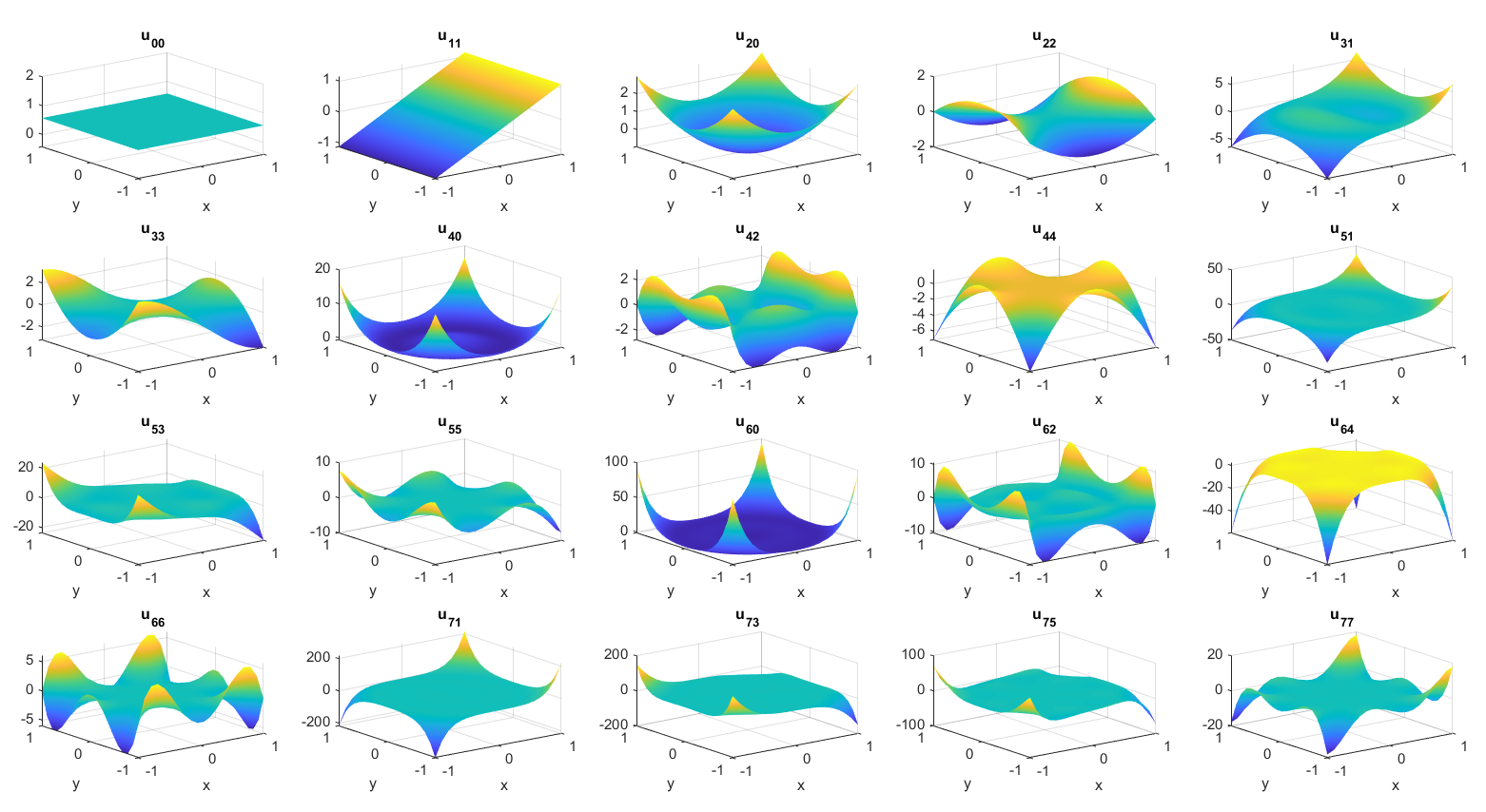}
	\caption{This plot visualizes the first twenty singular functions $u_{k,l}$, defined in \autoref{eq:vkl_ukl}, for $k =0,1,\ldots,7$, $l=0,1,\ldots,k$ and $l+k$ even.}
	\label{fig:ex_func_u}
\end{figure} 
We note that the first 49 non-zero eigenvalues of $R R^*$ belong to the following 12 distinct singular values 
\begin{equation} \label{eq:gamma}
	\Gamma= \set{\gamma_{0,0}, \gamma_{1,1}, \gamma_{2,0}, \gamma_{3,1}, \gamma_{4,0}, \gamma_{5,1}, \gamma_{6,0}, \gamma_{7,1}, \gamma_{8,0}, \gamma_{9,1}, \gamma_{10,0}, \gamma_{11,1}, \gamma_{12,0} }.
\end{equation}

\begin{example}[Learning spectral functions from analytical spectral functions] \label{eq:ex_1} 
	We start with a basic test by learning the spectral functions of the Radon-transform from input of the column vectors $\overline{y}_{k,l}:=v_{k,l}$, which are discretizations of the spectral functions of the Radon-transform as written down in \autoref{eq:vkl_ukl}. This should be the most simple test-case and the results should return discrete singular functions of the analytical singular functions. To be specific, we want to verify \autoref{th:proc} numerically. For this purpose, we construct the matrix $\overline{Y}$ consisting of the first $49$ singular functions (column vectors) $\overline{y}_{k,l}=v_{k,l}$ (some of them are shown in \autoref{fig:ex_func_v}) and the according matrix $A = \overline{Y} \overline{Y}^T \in \R^{\overline{m} \times \overline{m}}$ and $\overline{Y} \in \R^{\overline{m} \times 49}$, with $\overline{m}=2500$. In \autoref{fig:psi}, we show the first twenty results obtained from the singular value decomposition of $A$, graphically representing the eigenvectors of the matrix $\overline{W} = (\psi_1,\ldots,\psi_{49})$ as defined in \autoref{eq:ortho} in \autoref{th:proc}.
	\begin{figure}[h!]
		\centering
		\includegraphics[width=\textwidth]{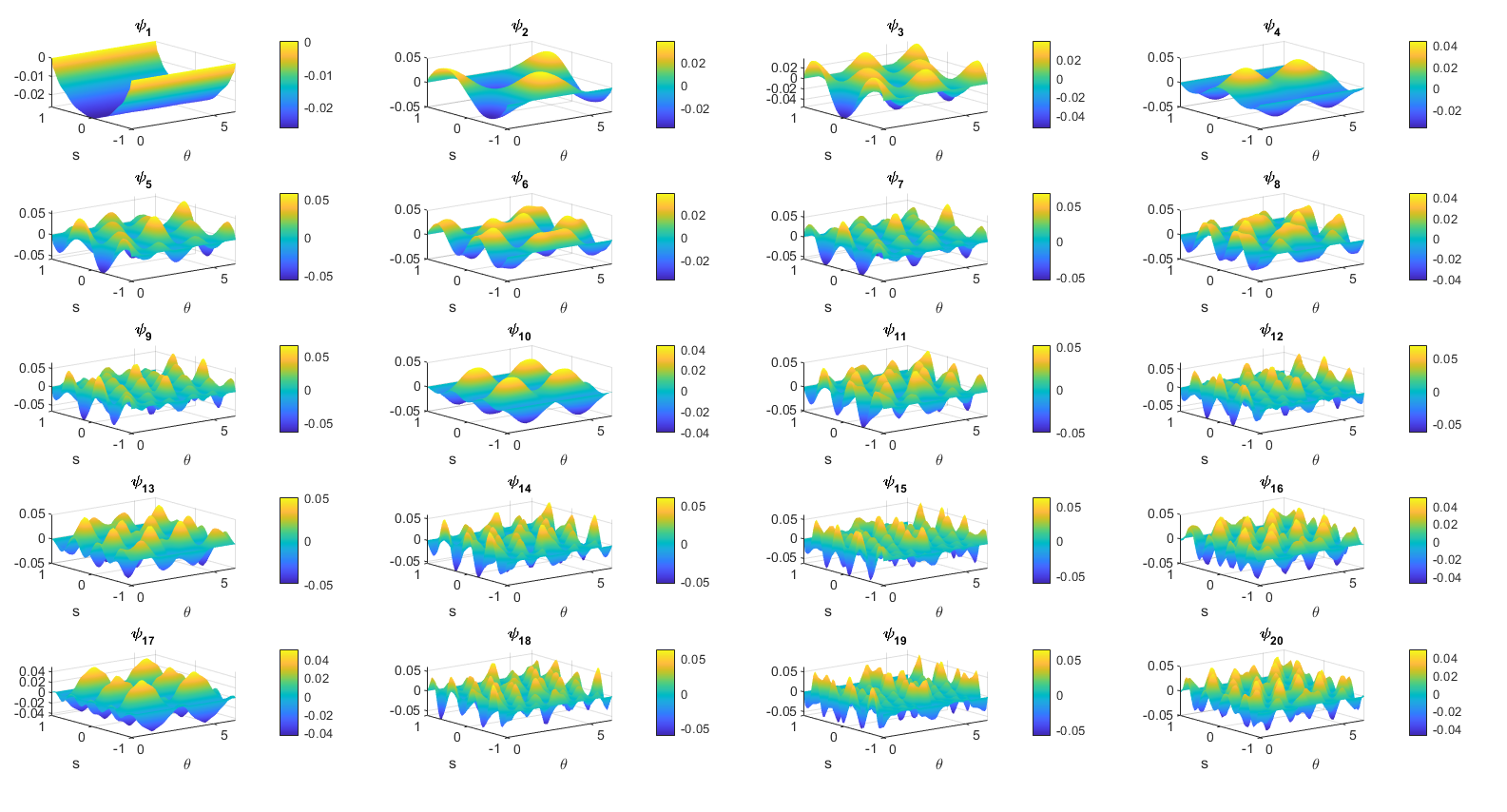}
		\caption{This plot shows the first twenty functions $\psi_h$ obtained from the singular value decomposition of the matrix $A$, defined in \autoref{eq:pcaa} with column vectors $\overline{y}_{kl}:=v_{kl}$. These results should be compared with $v_{k,l}$ as plotted in \autoref{fig:ex_func_v}. It is obvious that if a singular value has multiplicity $1$, then $v_{0,0}$ is also an eigenfunction, so $\psi_1$ perfectly represents the claim. For higher indices linear combinations are reconstructed approximately.}
		\label{fig:psi}
	\end{figure}
\end{example}

\begin{example} Comparing the results from analytical and learned singular value decomposition reveals what has been stated in \autoref{th:proc} that in case of eigenvalues of $RR^*$ with multiplicity higher that one only the according eigenspace can be reconstructed. In other words every $\psi_h$ should be able to be written as a linear combination of the according functions $v_{kl}$ with $(k,l) \in \mathcal{I}$ with some index $l$, which needs to be identified in addition. 
	
	Using the eigenspaces $E^k$, for $k=1,\ldots,13$, according to the spectral value $\gamma_k$, as defined in \autoref{th:proc}, we identify the sets to which the numerically calculated functions $\psi_{h}$, for $h=1,\ldots,49$ most likely belong. This is a two-step process as discussed below:
	
	\begin{enumerate}
		\item Identifying the set $E^k$, which $\psi_h$ most likely belongs to: We obtained with our Matlab code the classification given in \autoref{tab:Ek_psih}
		\begin{table}[ht]
			\centering
			\begin{tabular}{|c|c|}
				\hline
				\textbf{Set} & \textbf{Singular Values} \\
				\hline
				$E^1$ & $\psi_1$ \\
				\hline
				$E^2$ & $\psi_2$ \\
				\hline
				$E^3$ & $\psi_3, \psi_4$ \\
				\hline
				$E^4$ & $\psi_5, \psi_6$ \\
				\hline
				$E^5$ & $\psi_7$ \\
				\hline
				$E^6$ & $\psi_9$ \\
				\hline
				$E^7$ & $\psi_8, \psi_{10}$ \\
				\hline
				$E^8$ & $\psi_{11}, \psi_{13}$ \\
				\hline
				$E^9$ & $\psi_{12}, \psi_{14}, \psi_{16}, \psi_{17}$ \\
				\hline
				$E^{10}$ & $\psi_{15}, \psi_{18}, \psi_{20}, \psi_{22}, \psi_{23}, \psi_{45}, \psi_{46}$ \\
				\hline
				$E^{11}$ & $\psi_{19}, \psi_{21}, \psi_{24}, \psi_{25}, \psi_{26}, \psi_{27}, \psi_{44}, \psi_{47}, \psi_{48}, \psi_{49}$ \\
				\hline
				$E^{12}$ & $\psi_{28}, \psi_{30}, \psi_{31}, \psi_{33}, \psi_{34}, \psi_{39}, \psi_{41}, \psi_{43}$ \\
				\hline
				$E^{13}$ & $\psi_{29}, \psi_{32}, \psi_{35}, \psi_{36}, \psi_{37}, \psi_{38}, \psi_{40}, \psi_{42}$ \\
				\hline
			\end{tabular}
			\caption{Set $E^k$ which $\psi_h$ most likely belongs to.}\label{tab:Ek_psih}
		\end{table}
		\item Finding an appropriate linear combination of basis-elements of $E^k$, which optimally approximates $\psi_h$, for $h=1,\ldots,49$. For example, applying MATLAB's linear regression techniques, we obtained the following visual results, see Figure \ref{fig:psi_tog}, for the functions $\psi_h$, where $h=1,2,3,4,38,49$. Each plot displays: (1) the original function $\psi_h$ to be approximated, (2) the approximation using the best set of functions $E^k$ (see Table \ref{tab:Ek_psih}), and (3) the approximation error. For example, regarding the coefficients of the linear combinations for the functions $\psi_h$, where $h=1,2,3,4,38,49$, we have the results in Table \ref{tab:psi_lin_comb}.
		\begin{figure}[h!]
			\centering
			\includegraphics[width=\textwidth]{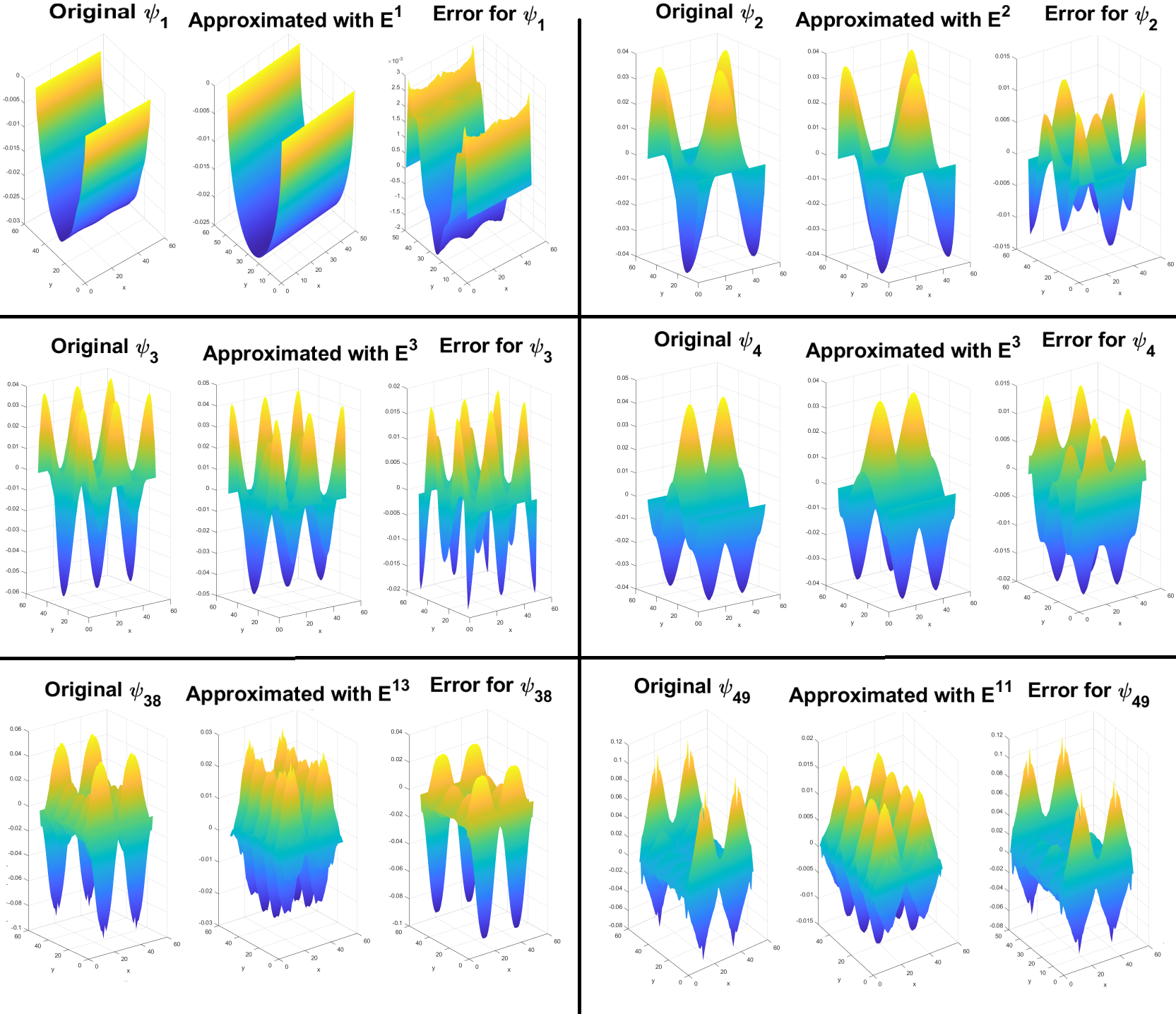}
			\caption{Results of the linear combinations for $\psi_{h}$, with $h=1,2,3,4,38,49$.}
			\label{fig:psi_tog}
		\end{figure}
		
		\begin{table}[ht]
			\centering
			\resizebox{\textwidth}{!}{	
				\begin{tabular}{|c|c|}
					\hline
					\textbf{Set} & \textbf{Linear Combination} \\
					\hline
					$E^1$ & $\psi_1=-0.02188 v_{0,0}$  \\
					\hline
					$E^2$ & $\psi_2=0.03360 v_{1,1}$ \\
					\hline
					$E^3$ & $\psi_3=0.01243 v_{2,0} + 0.03831 v_{2,2}$, $\psi_4=-0.03693 v_{2,0} + 0.01436 v_{2,2}
					$ \\
					\hline
					$E^{11}$ & $\psi_{49}=0.02541 v_{10,0}  -0.01122 v_{10,2}  -0.00060 v_{10,4}  -0.00072 v_{10,6}  -0.00061 v_{10,8} + 0.00003 v_{10,10}$ \\
					\hline
					$E^{13}$ & $\psi_{38}= 0.04555 v_{12,0} -0.02561 v_{12,2} -0.00396 v_{12,4} -0.00004 v_{12,6} +  0.00012 v_{12,8} + 0.00003 v_{12,10}  -0.00153 v_{12,12}$ \\
					\hline
				\end{tabular}
			}
			\caption{$\psi_h$, for $h=1,2,3,4,38,49$, as linear combinations of the eigenfunctions $v_{kl}$.}\label{tab:psi_lin_comb}
		\end{table}
		The numerical reconstructions of the singular functions is satisfying for the first twelve (the first four are shown in the first and second line of \autoref{fig:psi_tog}). After that the errors are pronounced, with values of the same order of magnitude as those of the functions $\psi_h$ (see, for example, the plots of the last line in \autoref{fig:psi_tog}). 
	\end{enumerate}
	The example has shown that reconstruction of the singular functions and values of operators is possible from training data without knowing the operator explicitly. But the higher oscillating the functions are the more complicated it gets numerically.
\end{example}

\begin{example}[Learning spectral functions from arbitrary test functions] \label{eq:ex_1a} 
	In this example, we consider a dataset, composed of twenty elements, that we have constructed by ourself using a Matlab function. Specifically, each discrete image, of numerical dimension $50\times50$, is composed of ten ellipses that have been generated randomly. In fact, we want to create several images that mimic the structure of the Shepp-Logan phantom. Each ellipse is described by assigning seven parameters (intensity, length of the major semi-axis, length of the minor semi-axis, x-coordinate of the center, y-coordinate of the center, rotation angle) that vary within the ranges reported in the \autoref{tab:par_ellipses}.
	\begin{table}[h!]
		\centering
		\begin{tabular}{|l|l|}
			\hline
			\textbf{Parameter of Ellipse} & \textbf{Range} \\ \hline
			Intensity & $[0.01, 2]$ \\ \hline
			Length of the major semi-axis & $[0.1, 0.7]$ \\ \hline
			Length of the minor semi-axis & $[0.1, 0.7]$ \\ \hline
			x-coordinate of the center & $[-0.6, 0.6]$ \\ \hline
			y-coordinate of the center & $[-0.6, 0.6]$ \\ \hline
			Rotation angle & $[-45^\circ, 134^\circ]$ \\ \hline
		\end{tabular}
		\caption{Parameters and their corresponding ranges for the ellipses of the dataset described in \autoref{eq:ex_1a}.} \label{tab:par_ellipses}
	\end{table}
	In \autoref{fig:data} and \autoref{fig:ortho_data_above}, we present the plot of the images constructed with our MATLAB routine and their orthonormalized version ( in \autoref{fig:ortho_data_above}, we show the projection of the values onto the $xy$ plane for better visualization of the orthonormalization procedure).
	\begin{figure}[h!]
		\centering
		\centering
		\includegraphics[width=\textwidth]{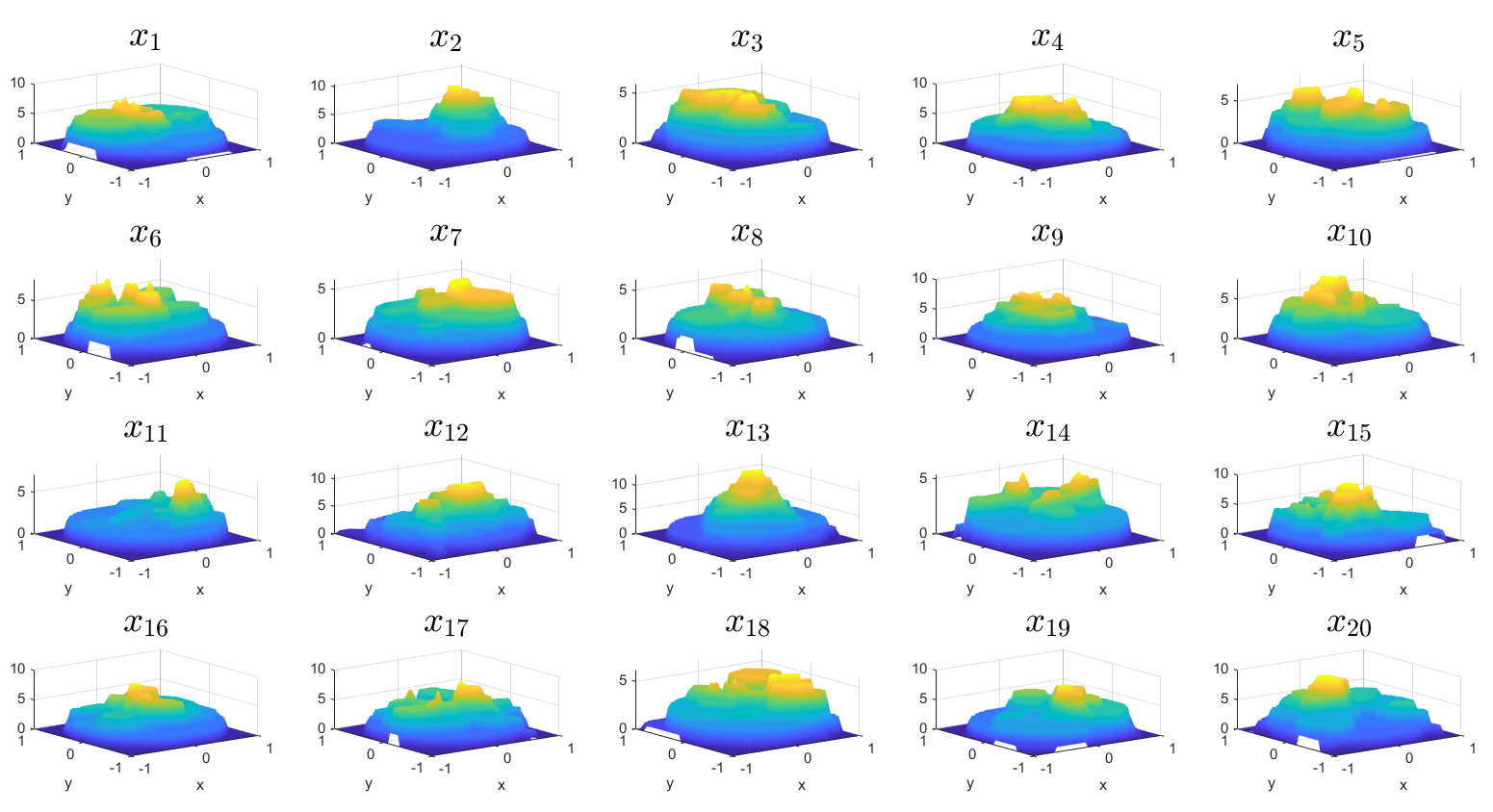} 
		\caption{Images $x_k$, for $k=1,\ldots,20$.}
		\label{fig:data}
	\end{figure}
	\begin{figure}
		\centering
		\includegraphics[width=\textwidth]{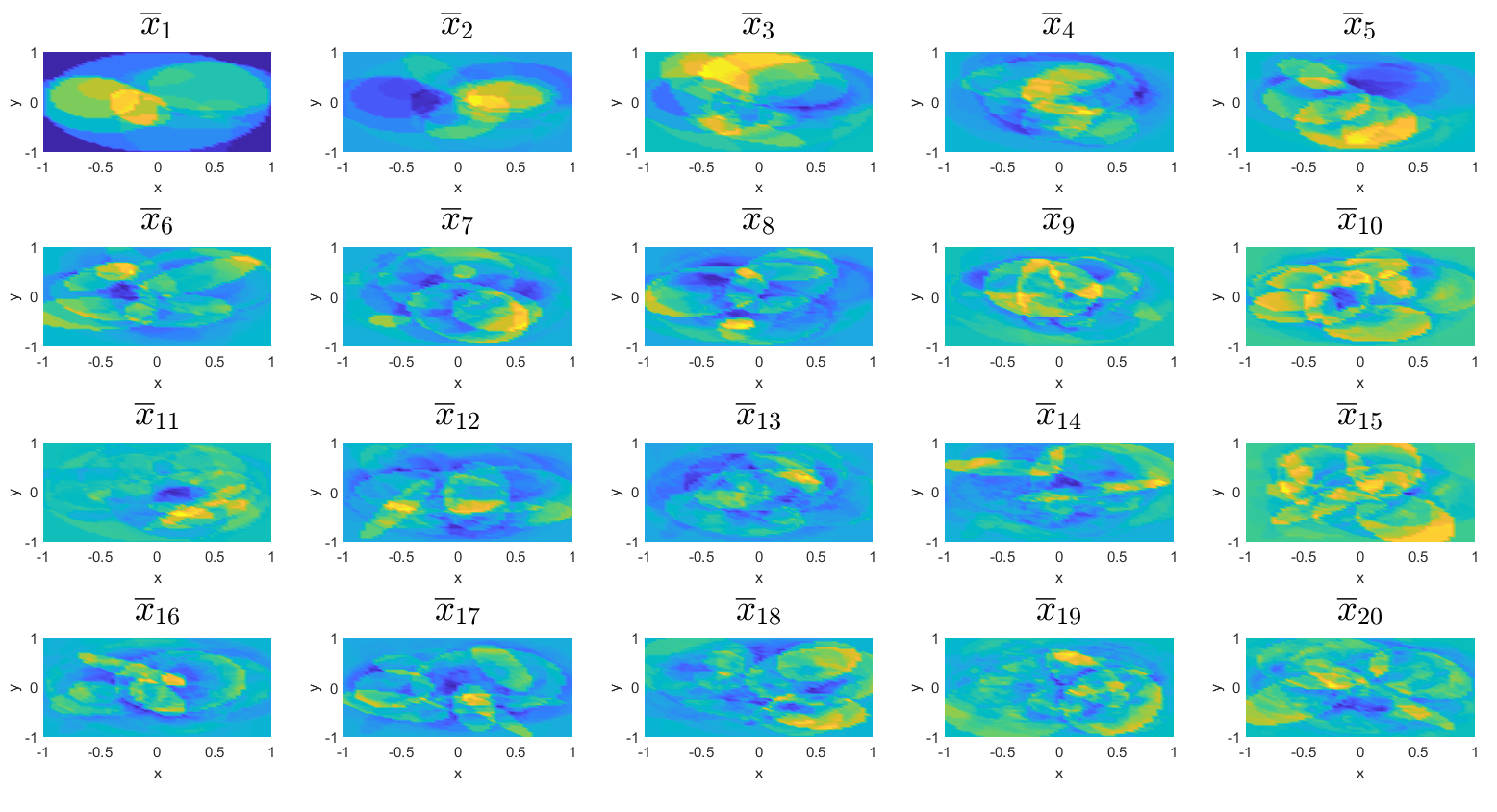}
		\caption{Orthonormalized images $\overline{x}_k$, for $k=1,\ldots,20$, projected onto the plane $xy$.}
		\label{fig:ortho_data_above}
	\end{figure}
	The corresponding data (that is the sinograms) of the orthonormalized images, see \autoref{eq:gsi}, $\overline{y}_k=R(\overline{x}_k)$ are plotted in \autoref{fig:sinograms}.
	\begin{figure}
		\centering
		\includegraphics[width=\textwidth]{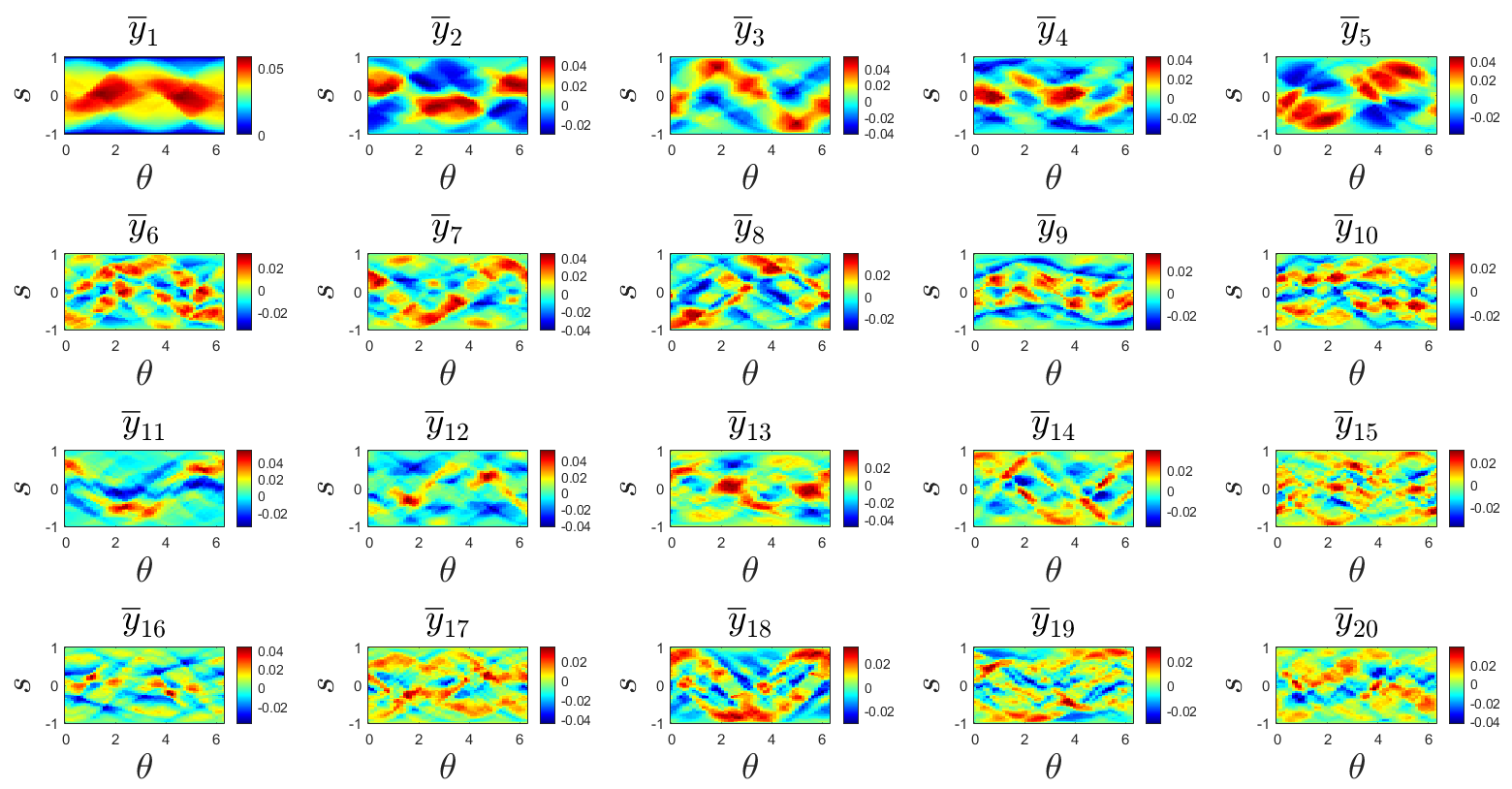}
		\caption{Data $\overline{y}_k$, for $k=1,\ldots,20$.}
		\label{fig:sinograms}
	\end{figure}
	As in the previous numerical example, see \autoref{eq:ex_1}, we utilize the eigenspaces $E^k$ for $k=1,\ldots,8$, to identify the sets to which the numerically computed functions 
	$\psi_h$, for $h=1,\ldots,20$ most likely belong. 
	We obtained with our Matlab code the classification given in \autoref{tab:Ek_psih_sl}.
	\begin{table}[ht]
		\centering
		\begin{tabular}{|c|c|}
			\hline
			\textbf{Set} & \textbf{Singular Values} \\
			\hline
			$E^1$ & $\psi_1$ \\
			\hline
			$E^2$ & $\psi_2, \psi_3$ \\
			\hline
			$E^3$ & $\psi_4, \psi_5, \psi_6$ \\
			\hline
			$E^4$ & $\psi_9$ \\
			\hline
			$E^5$ & $\psi_{10}$ \\
			\hline
			$E^6$ & $\psi_{7}, \psi_{8}, \psi_{18}, \psi_{19}$ \\
			\hline
			$E^7$ & $\psi_{11}, \psi_{12}, \psi_{13}, \psi_{14}, \psi_{17}, \psi_{20}$ \\
			\hline
			$E^8$ & $\psi_{15}, \psi_{16}$ \\
			\hline
		\end{tabular}
		\caption{Set $E^k$ which $\psi_h$ most likely belongs to in the case of \autoref{eq:ex_1a}.}\label{tab:Ek_psih_sl}
	\end{table}
	As before, we identify suitable linear combinations of elements of $E^k$ in order to rewrite the term $\psi_h$. We present in \autoref{fig:psi_shepp_logan} the first four results of the linear combinations, allowing for a direct comparison with the outcomes previously obtained in \autoref{eq:ex_1}.  
	\begin{figure}[h!]
		\centering
		\includegraphics[width=0.48\textwidth]{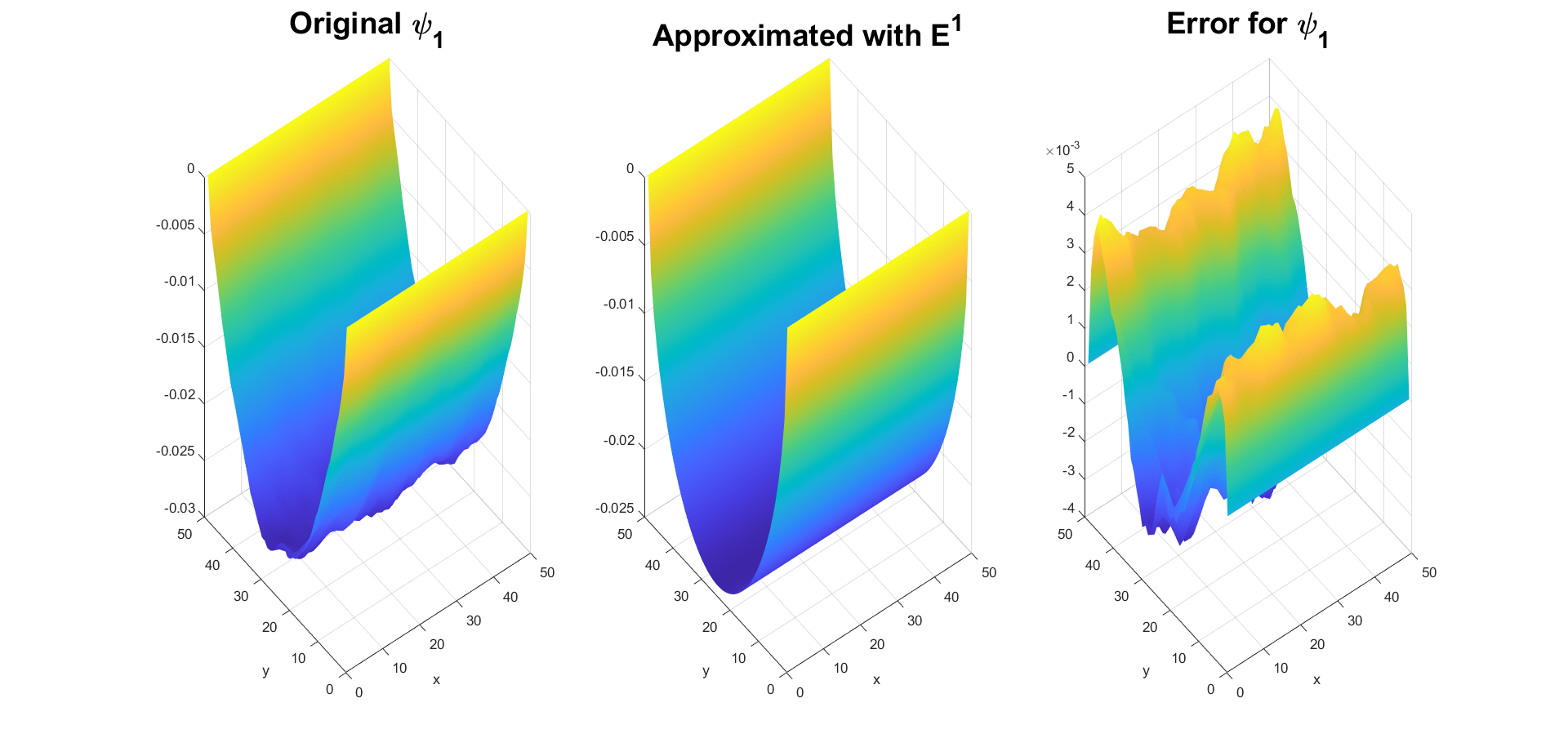}
		\includegraphics[width=0.48\textwidth]{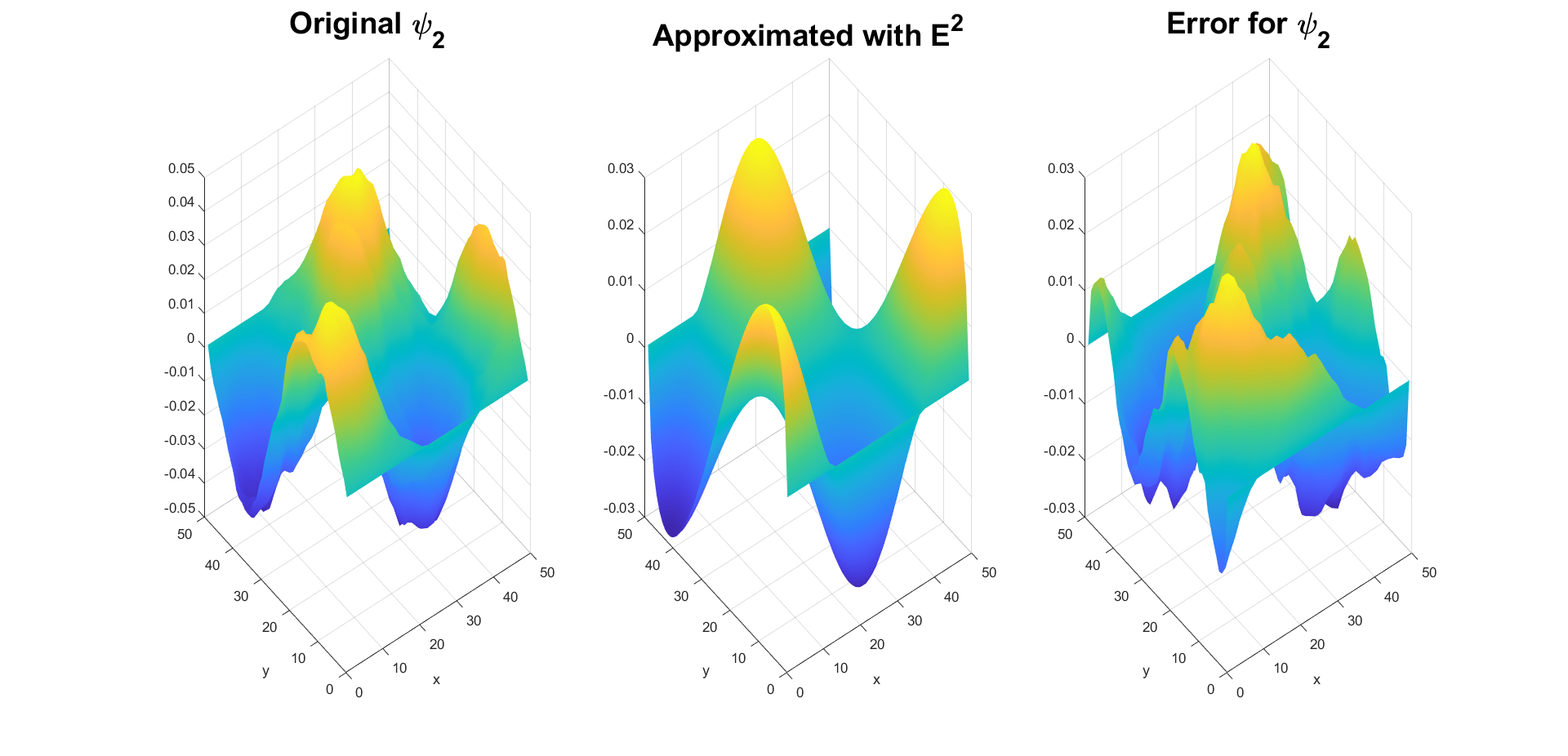} \\
		\includegraphics[width=0.48\textwidth]{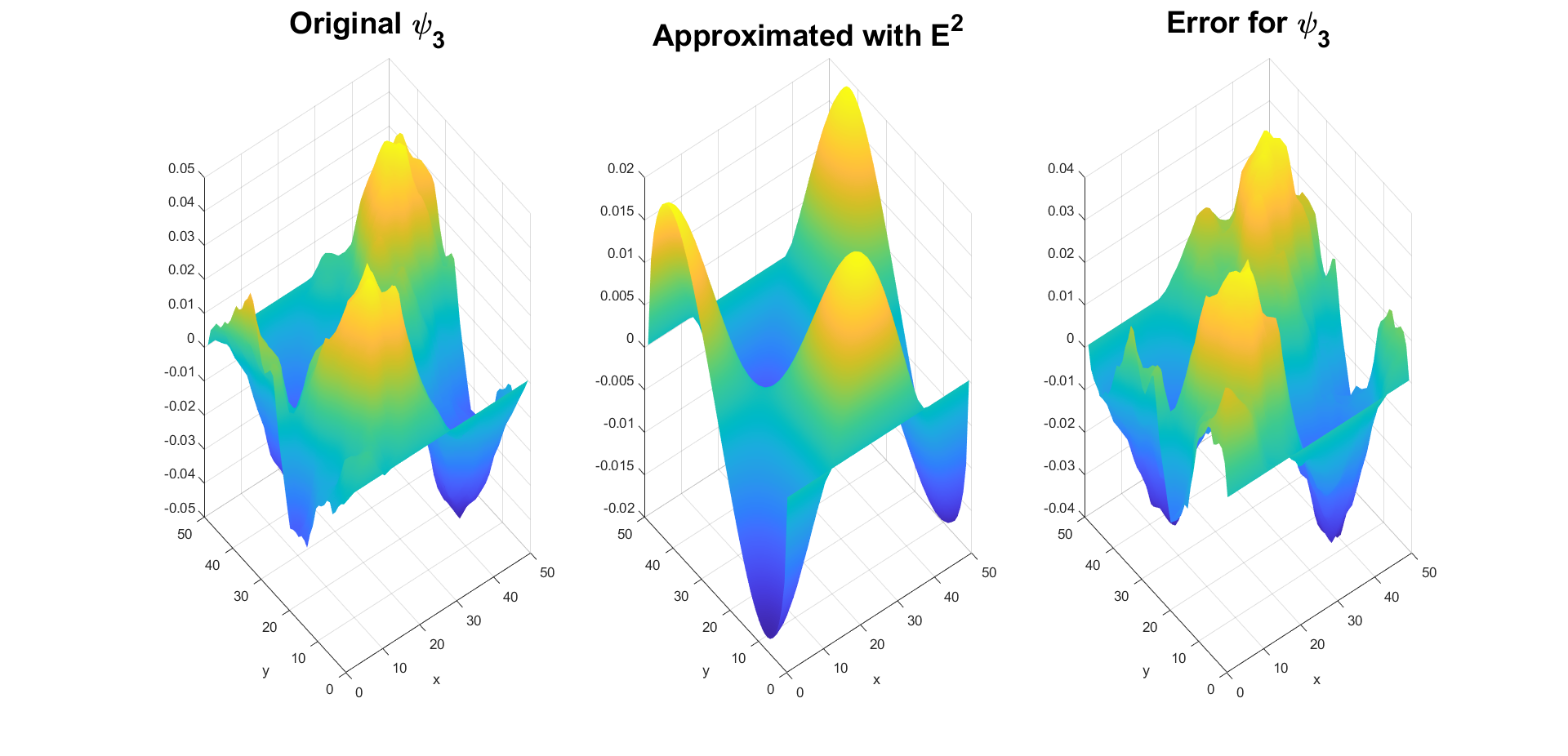}
		\includegraphics[width=0.48\textwidth]{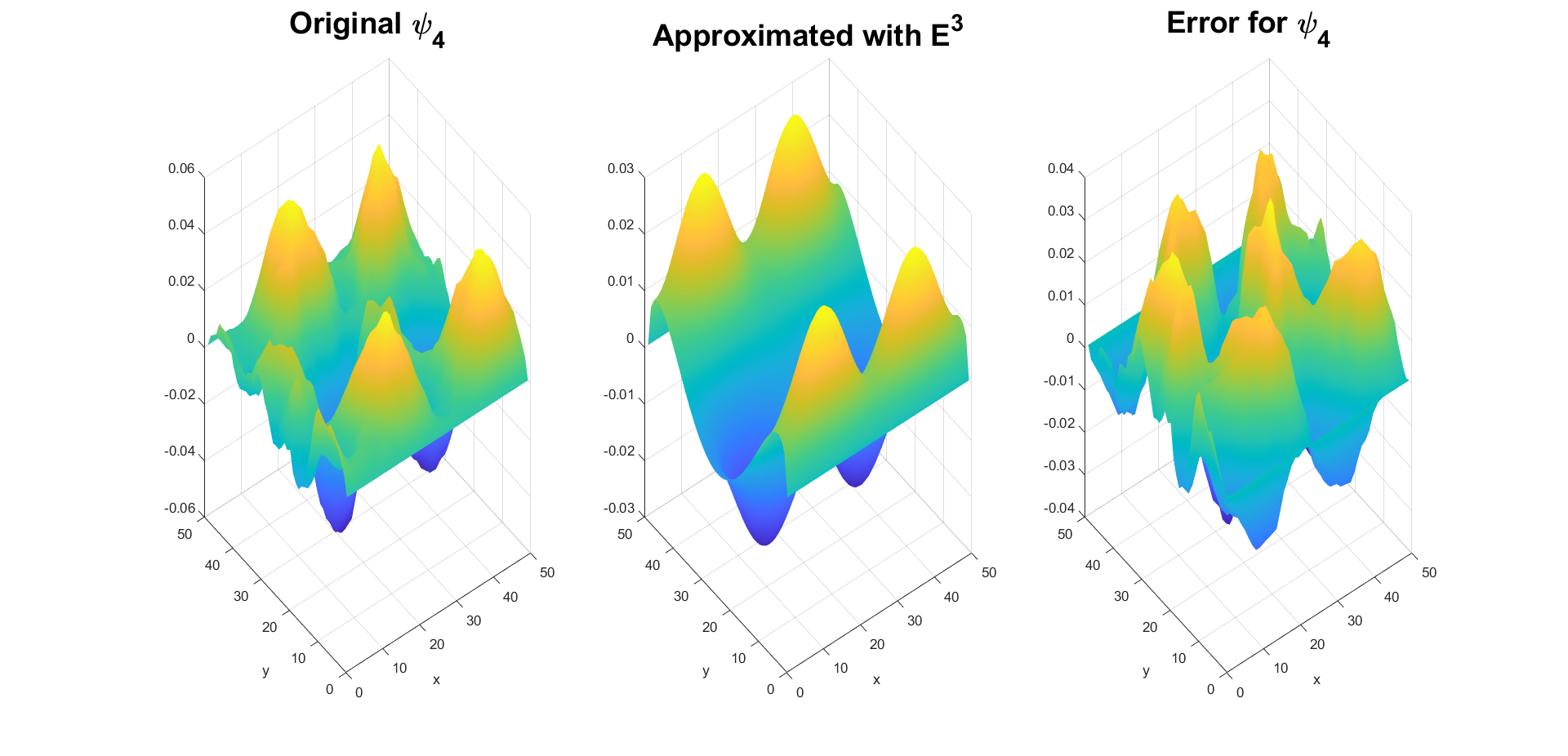}
		\caption{Results of the linear combinations for $\psi_{h}$, with $h=1,2,3,4$.}
		\label{fig:psi_shepp_logan}
	\end{figure}
	The reader can observe, as expected, that the approximation results, expressed as linear combinations, are less precise than those found in \autoref{eq:ex_1}. 
\end{example} 

\subsubsection{Decoding}
Using the results from the previous subsection, we implement the reconstruction formula derived in \autoref{eq:decoder}. 
\begin{description}
	\item{\bf Test 1:}  In the first example the ground truth is given by 
	\begin{equation} \label{eq:test1} 
		\begin{aligned}
			\bx_{true} = -0.9119\bu_{11,5} +0.6527\bu_{3,1} &-0.7343\bu_{6,4}+ 0.5406\bu_{5,3} \\
			&+ 0.9758 \bu_{8,4} -0.1569 \bu_{12,10} + 0.2778  \bu_{7,3} + 0.6395  \bu_{12,2}.
		\end{aligned}
	\end{equation}
	In \autoref{fig:xls} the numerical reconstruction with the decoder \autoref{eq:decoder} is plotted. It provides a reasonable reconstruction.
	\begin{figure}[h!]
		\centering
		\includegraphics[width=0.3\textwidth]{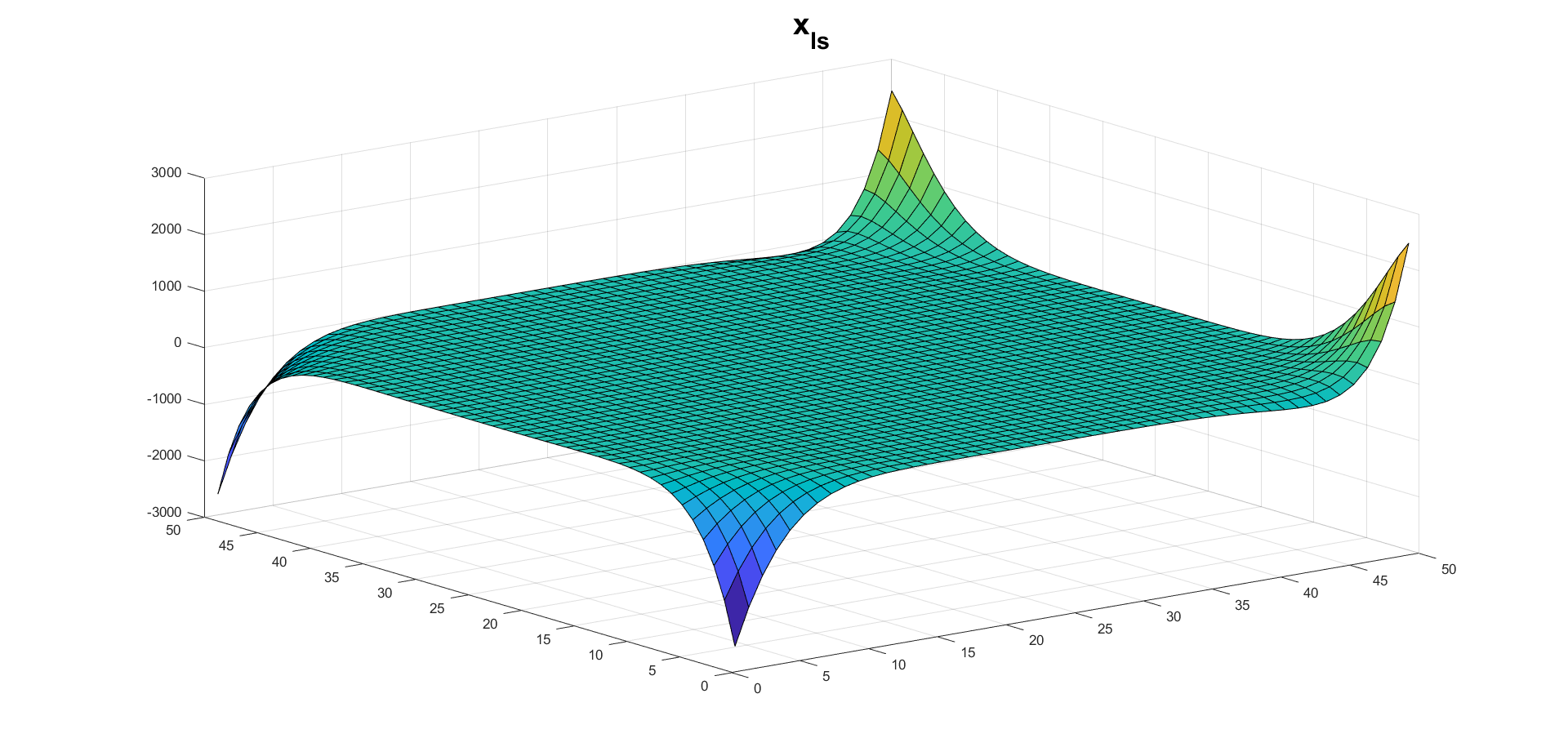} \quad \includegraphics[width=0.3\textwidth]{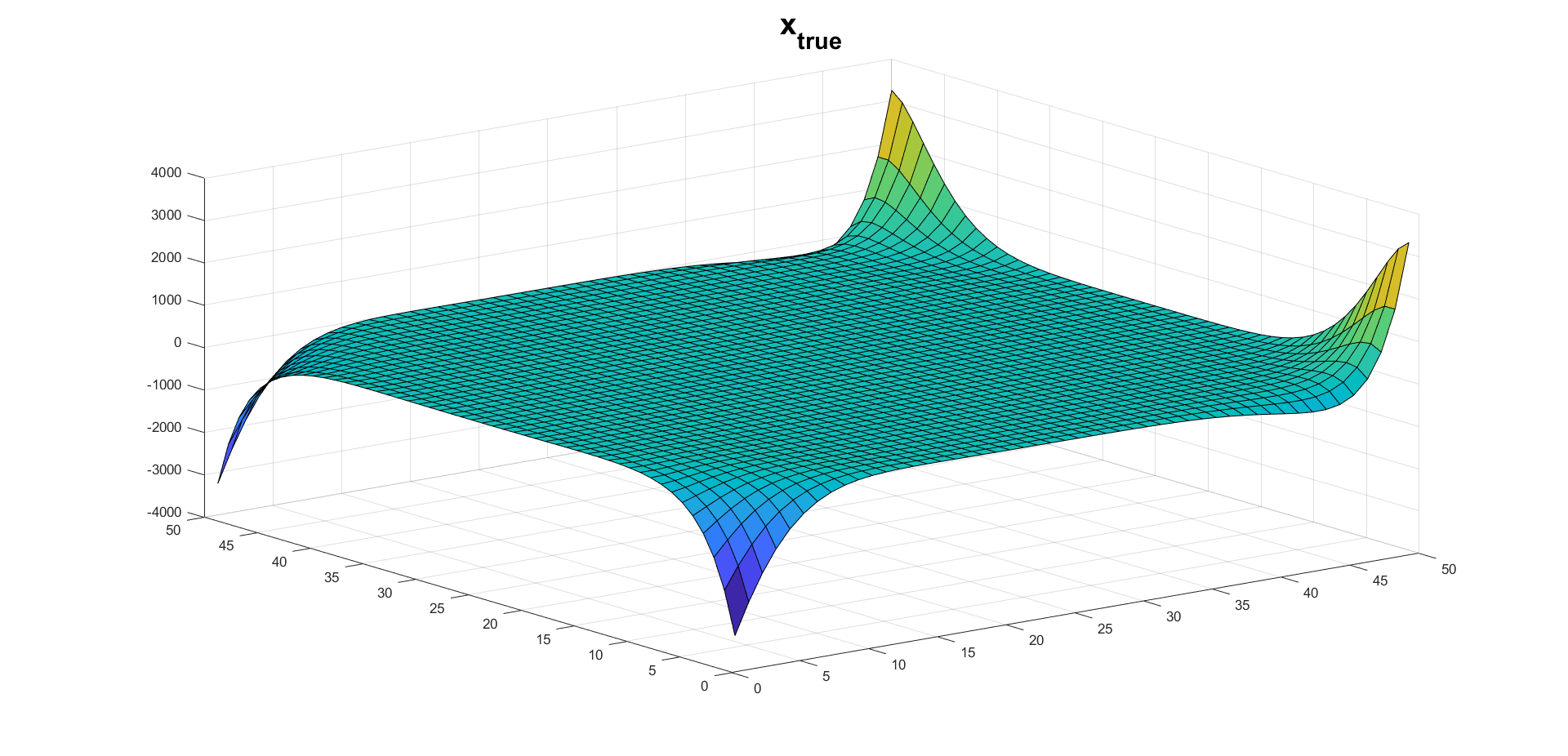} \quad \includegraphics[width=0.3\textwidth]{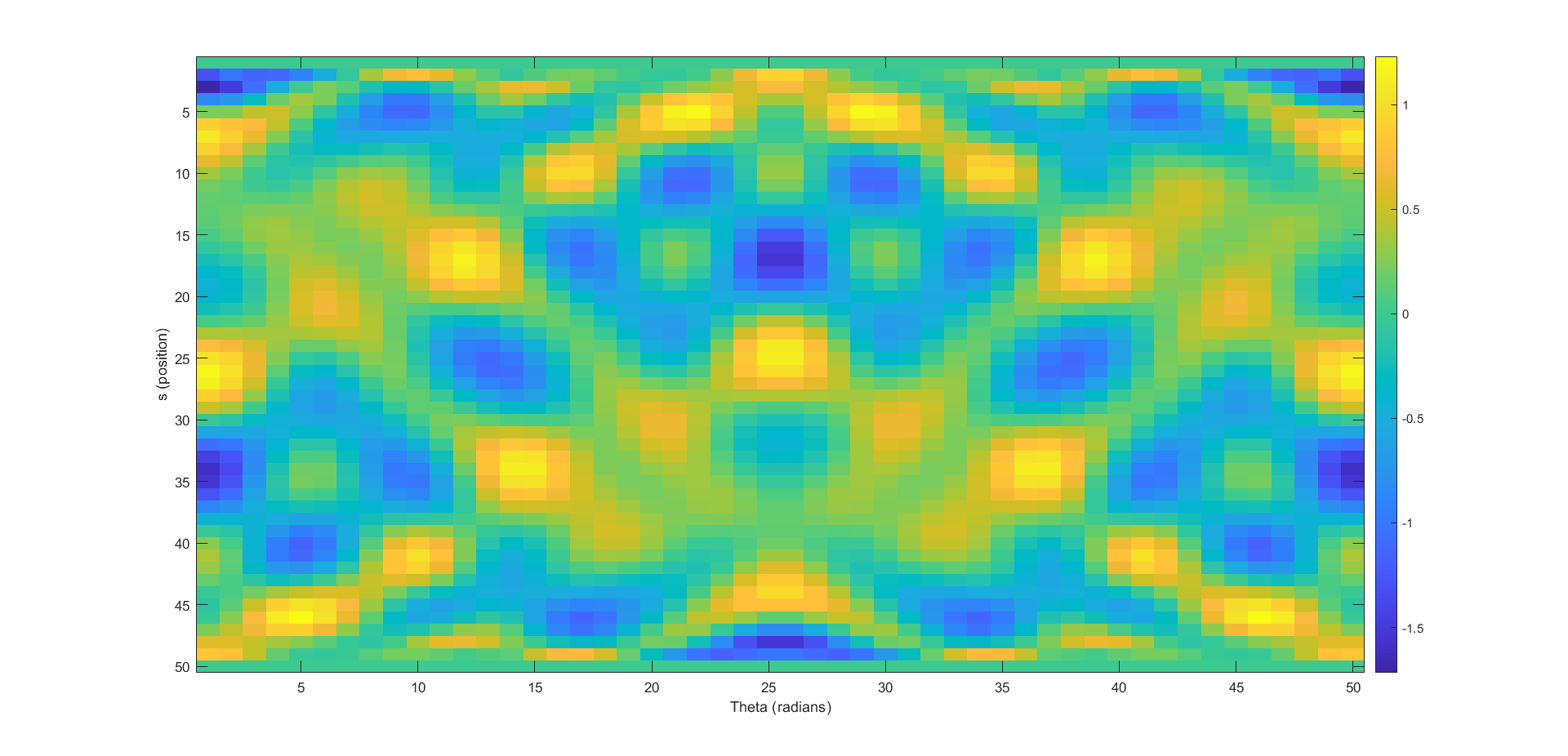}
		\caption{Decoding, ground truth as represented in \autoref{eq:test1} and $y_{true}$.}
		\label{fig:xls}
	\end{figure}
	\item{\textbf{Test 2}}. We introduce a scenario similar to \textbf{Test 1}, but now involving a nonlinear combination of functions $\bu_{k,l}$. Specifically, for each function $\bu_{k,l}$ in the nonlinear combination, we apply the operation $0.1 \bu_{k,l}^2+\e^{\bu_{k,l}/\max(\bu_{k,l})}$, scaled by a specific weight. Results and comparison with ground truth are provided in \autoref{fig:xls_nl2}. As expected, in the nonlinear setting the results are worse.
	\begin{figure}[h!]
		\centering
		\includegraphics[width=0.3\textwidth]{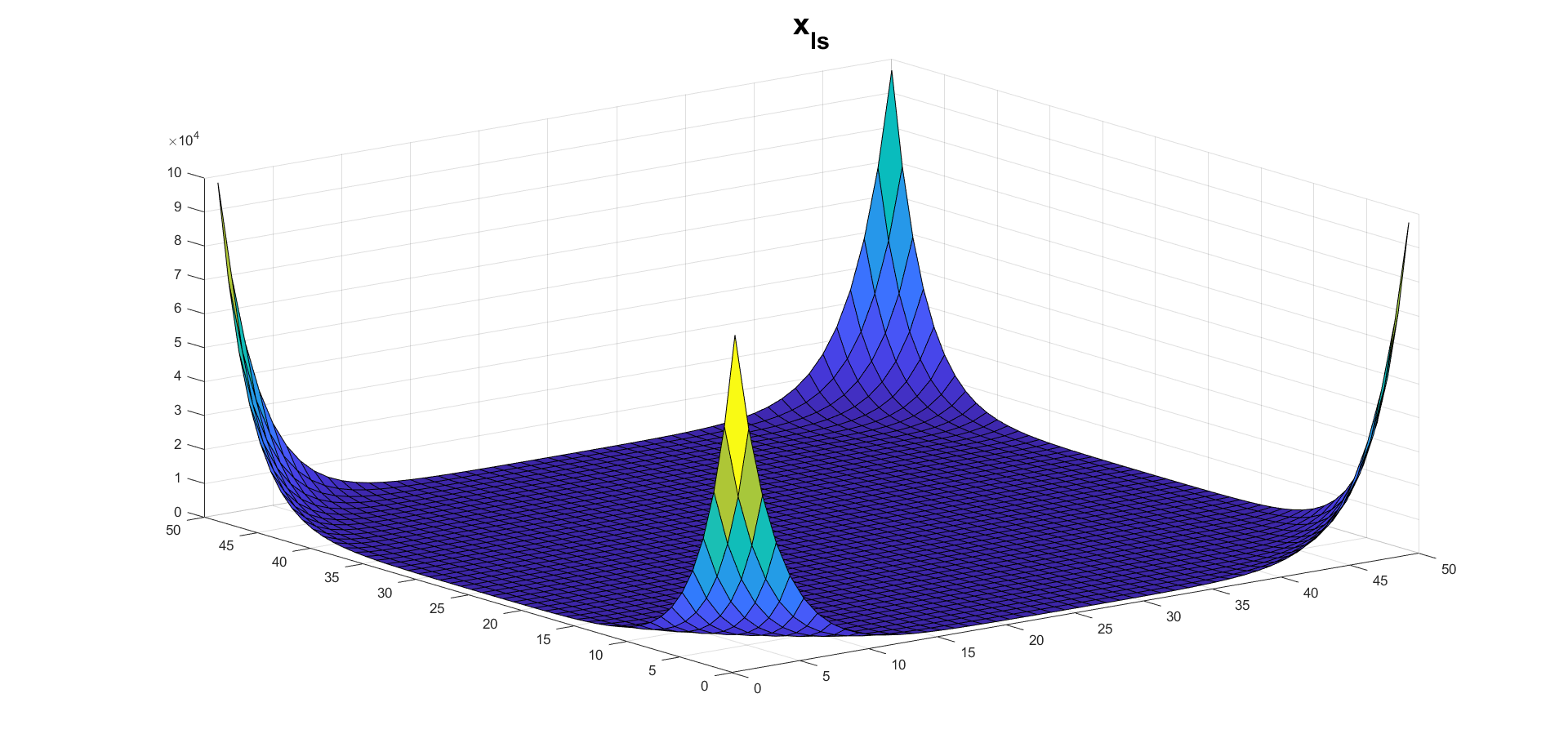} \quad \includegraphics[width=0.3\textwidth]{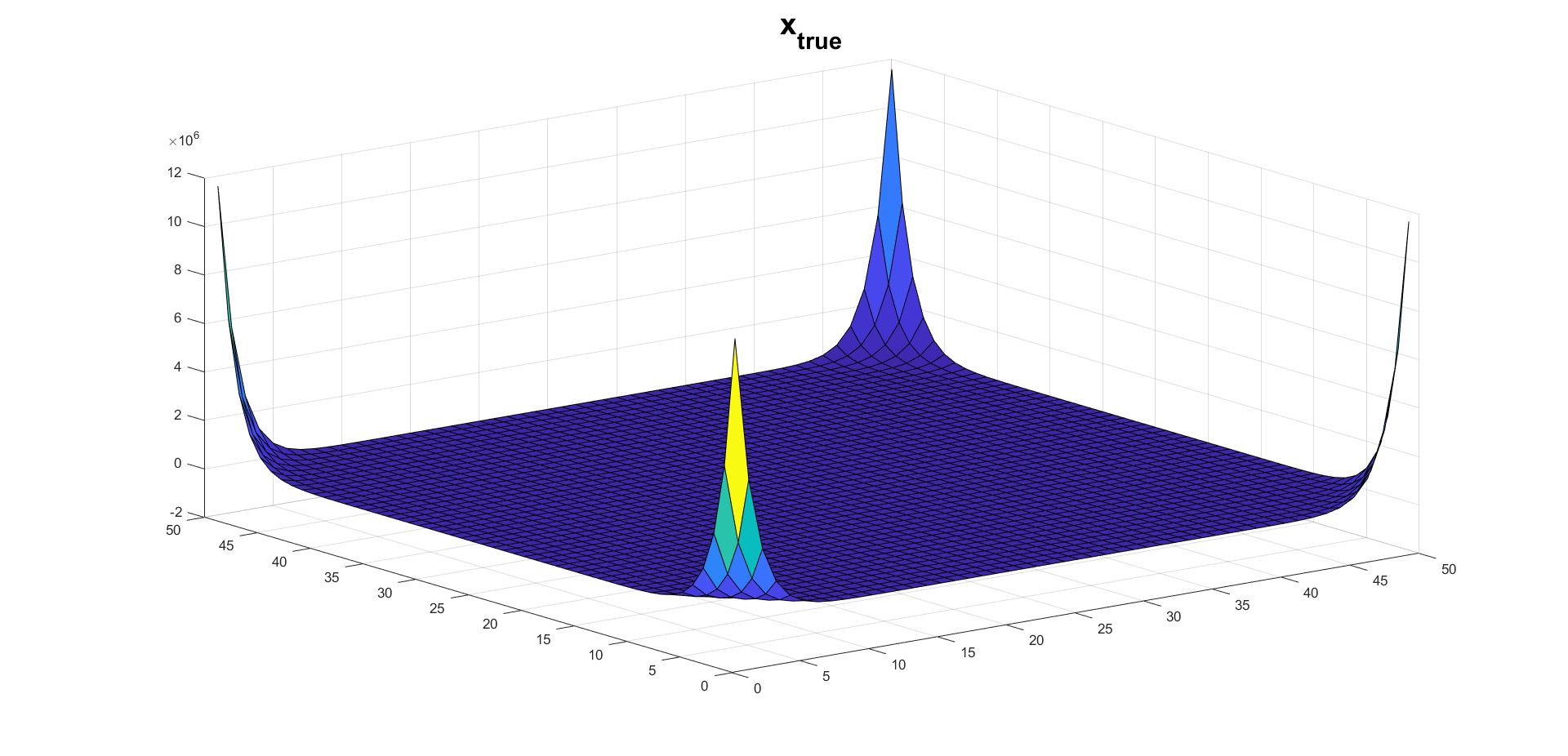} \quad \includegraphics[width=0.3\textwidth]{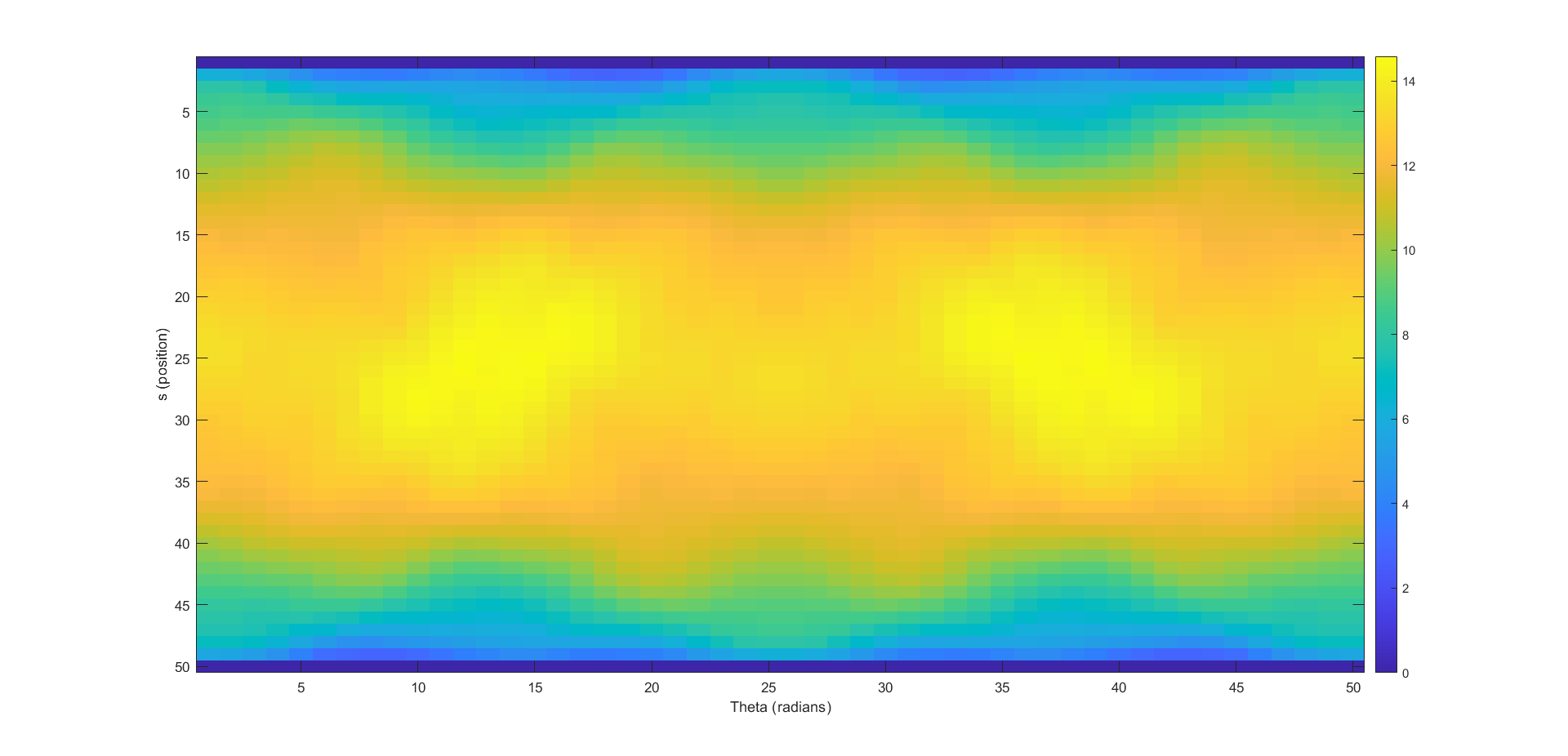}
		\caption{Decoding, ground truth as represented in \autoref{eq:test1} and $y_{true}$ in the nonlinear combination setting.}
		\label{fig:xls_nl2}
	\end{figure}
\end{description}

\section*{Conclusion} In this paper we have shown that a double orthonormalization strategy, consisting first of orthonormalization of training images and second of a principal component analysis of the training data theoretically provides us with the singular value decomposition of a linear operator (without making use of an explicit physical model). Two further interesting aspects are shown: First: Orthonormalization, like Gram-Schmidt, can be expressed as a deep neural network, and this opens up exploiting synergies with de-and encoder strategies. Secondly, not only the singular value decomposition can be implemented purely data driven but also the decoding, meaning that for arbitrary data the minimum-norm solution can be computed data driven.

\appendix
\subsection*{Acknowledgements}
This research was funded in whole, or in part, by the Austrian Science Fund
(FWF) P 34981 -- New Inverse Problems of Super-Resolved Microscopy (NIPSUM).
For the purpose of open access, the author has applied a CC BY public copyright
license to any Author Accepted Manuscript version arising from this submission.
Moreover, OS is supported by the Austrian Science Fund (FWF),
with SFB F68 ``Tomography Across the Scales'', project F6807-N36 (Tomography with Uncertainties).
The financial support by the Austrian Federal Ministry for Digital and Economic
Affairs, the National Foundation for Research, Technology and Development and the Christian Doppler
Research Association is gratefully acknowledged. AA is member of the group GNAMPA 
(Gruppo Nazionale per l’Analisi Matematica, la Probabilità e le loro Applicazioni) of INdAM (Istituto Nazionale di Alta Matematica).
The computational results presented have been achieved in part using the Vienna Scientific Cluster (VSC).

The authors want to thank Martin Rumpf (University of Bonn) for stimulating discussions and feedback.

\section*{References}
\renewcommand{\i}{\ii}
\printbibliography[heading=none]

\end{document}